\documentclass[10pt]{amsart} 

\newtheorem{definition}{Definition}[section]
\newtheorem{lemma}[definition]{Lemma}
\newtheorem{theorem}[definition]{Theorem}
\newtheorem{proposition}[definition]{Proposition}
\newtheorem{corollary}[definition]{Corollary}
\newtheorem{remark}[definition]{Remark} 
\numberwithin{equation}{section}
\newcommand \be     {\begin{equation}}
\newcommand \ee     {\end{equation}}
\newcommand \RR     {\mathbb{R}}
\newcommand \RN     {\mathbb{R}^N}
\newcommand \eps    {\epsilon}  
\newcommand \del    {{\partial}}

\newcommand \Bzero 	{\Bcal_{\delta_0}}
\newcommand \Bone 	{\Bcal_{\delta_1}}
\newcommand \Btwo 	{\Bcal_{\delta_2}}

\newcommand \Acal   {{\mathcal A}} 
\newcommand \Bcal   {{\mathcal B}} 
\newcommand \Ccal   {{\mathcal C}} 
\newcommand \Dcal   {{\mathcal D}} 
\newcommand \Jcal   {{\mathcal J}} 
\newcommand \Fcal   {{\mathcal F}}
\newcommand \Hcal   {{\mathcal H}}
\newcommand \Ical   {{\mathcal I}}
\newcommand \Kcal   {{\mathcal K}}
\newcommand \Lcal   {{\mathcal L}}
\newcommand \Ncal   {{\mathcal N}}
\newcommand \Rcal   {{\mathcal R}}
\newcommand \Scal   {{\mathcal S}} 
\newcommand \Tcal   {{\mathcal T}}
\newcommand \Ucal   {{\mathcal U}}
\newcommand \Wcal   {{\mathcal W}}
\newcommand \lam    \lambda 
\newcommand \Lam    \Lambda 
\newcommand \lamb   {\overline{\lambda}}  
\newcommand \Lamb   {\overline{\Lambda}}  
\newcommand \lb {\overline l} 
\newcommand \vs     {{\mathbf s}}
\newcommand \va     {{\mathbf a}}
\newcommand \vv     {{\mathbf v}}
\newcommand \vw     {{\mathbf w}}
\newcommand \ubf    {{\mathbf u}}
\newcommand \vn     {{\mathbf n}}
\newcommand \fbf    {{\mathbf f}}
\newcommand \gbf    {{\mathbf g}}

\begin{document}
\bibliographystyle{plain}
\title[Why many theories of shock waves are necessary]
{Why many theories of shock waves are necessary. Kinetic relations for nonconservative systems
}
\author[C. Berthon, F. Coquel, and P.G. L{\tiny e}Floch]{Christophe Berthon$^1$, Fr\'ed\'eric Coquel$^2$,
\\ 
\and 
\\
Philippe G. L{\smaller e}Floch$^2$
}
\thanks{Completed on May 30, 2010. Revised on February 15, 2011. To appear in Proc. Royal Soc. Edinburgh.
\newline 
$^1$ Laboratoire de Math\'ematiques Jean Leray, Centre National de la Recherche Scientifique,  
Universit\'e de Nantes, 2 rue de la Houssini\`ere, BP 92208, 44322 Nantes, France.
\newline 
Email: {\tt Christophe.Berthon@math.univ-nantes.fr.}
\newline
$^2$ Laboratoire Jacques-Louis Lions, Centre National de la Recherche Scientifique,   
Universit\'e Pierre et Marie Curie (Paris 6), 75252 Paris, France.
\newline 
Email: {\tt Coquel@ann.jussieu.fr, pgLeFloch@gmail.com.}
\newline 
\textit{\ AMS Subject Class.} 35L60, 76B15, 74N20. 
\textit{Key Words and Phrases.} Nonconservative hyperbolic system, kinetic relation, Riemann problem, 
multi-pressure Navier-Stokes equations, traveling wave.}
\date{}

\begin{abstract}
For a class of nonconservative hyperbolic systems of partial differential equations
endowed with a strictly convex mathematical entropy, 
we formulate the initial value problem by supplementing the equations with a kinetic relation
prescribing the rate of entropy dissipation across shock waves. 
Our condition can be regarded as a generalization to nonconservative systems of 
a similar concept introduced by Abeyaratne, Knowles, and Truskinovsky 
for subsonic phase transitions and 
by LeFloch for nonclassical undercompressive shocks to nonlinear hyperbolic systems. 
The proposed kinetic relation for nonconservative systems turns out to be equivalent,
for the class of systems under consideration at least, 
to Dal~Maso, LeFloch, and Murat's definition based on a prescribed family of Lipschitz continuous paths. 
In agreement with previous theories, the kinetic relation should be derived from a phase plane analysis 
of traveling solutions associated with an augmented version of the nonconservative system.   
We illustrate with several examples that nonconservative systems arising in the 
applications fit in our framework, and for a typical model of turbulent fluid dynamics, 
we provide a detailed analysis of the existence and properties of traveling waves which yields the 
corresponding kinetic function.  
\end{abstract}
\maketitle


\section{Introduction}
\label{1-0} 

Certain nonlinear hyperbolic models arising in continuum physics and, especially, models 
describing complex fluid flows do not take the standard form of conservation laws 
but, instead, are nonlinear hyperbolic systems in nonconservative form
\be
\del_t u + A(u) \, \del_x u = 0, \qquad \, x \in \RR, \quad t \geq 0. 
\label{1.1}
\ee 
Here, $u=u(x,t) \in \Omega$ is an unknown field taking values in a (convex and open) domain $\Omega \subset \RN$,  
while the matrix-valued field $A=A(u)$ is given and, for each state $u$, admits 
$N$ real and distinct eigenvalues. 
It is well-known that nonlinear hyperbolic equations do not admit smooth solutions since propagating discontinuities 
arise in finite time even from smooth initial data. For conservative systems, weak solutions in the sense of distributions are sought for. However, for nonconservative systems \eqref{1.1}, the distributional definition does not apply.
A suitable notion of weak solution was proposed Dal~Maso, LeFloch, and Murat \cite{DLM}
and the nonlinear stability of such solutions was investigated therein. Nonconservative hyperbolic systems 
have been the subject of active research in the past fifteen years. The theory 
covers the definition of weak solutions 
\cite{Volpert,LeFloch-CPDE,LeFloch-IMA,DLM,LeFloch-JHDE}, 
the existence of solutions to the Riemann problem \cite{LeFloch-CPDE,DLM}
the initial value problem \cite{LeFloch-CPDE,LeFlochLiu,CrastaLeFloch},  
the uniqueness of bounded variation solutions \cite{BLP,LeFloch-book}, 
and their approximation via finite difference schemes \cite{HouLeFloch,CLMP,LM}.  
In addition, many nonconservative models arising in continuum mechanics have been 
systematically investigated, 
as such model play an important role in the modeling of multi-phase flows and turbulent fluid dynamics
\cite{AudebertCoquel,BerthonCoquel-2006,BerthonCoquel1, BerthonCoquel2,ChalonsCoquel-2006, ChalonsCoquel-2007}. 

Building upon the above works, our purpose in the present paper
is to consider a restricted class of nonconservative systems of the form \eqref{1.1}, 
characterized by the property that a large family of additional entropy functions (conservation laws) 
is also available. In other words, the systems to be considered below {\sl formally} have a conservative form 
if nonlinear combinations of the given equations are allowed. However, the physical modeling
dictate that nonconservative equations should be kept and it is precisely under these conditions that 
a ``kinetic relation'', as we propose in the present paper, should enter into play. 

The kinetic relation were initially introduced by LeFloch~\cite{LeFloch-review} for hyperbolic 
systems of conservation laws in order handle nonclassical undercompressive shocks, 
following earlier works by Abeyaratne, Knowles \cite{AbeyaratneKnowles}, and Truskinovsky 
\cite{Truskinovsky} for subsonic phase transitions. 
See \cite{LeFloch-ARMA,LeFloch-review, LeFloch-book, LeFloch-rev2} for details.  

The concept of a kinetic relation for nonconservative systems discussed now 
was actually introduced first by the authors in an unpublished manuscript~\cite{BerthonCoquelLeFloch}. 
Later on, this concept was investigated numerically 
by Aubert, Berthon, Chalons, and Coquel \cite{AudebertCoquel,BerthonCoquel2,ChalonsCoquel-2007}, 
and the control of the numerical dissipation of finite difference schemes was extensively addressed. 
Our purpose in the present paper is to provide the required theoretical framework
and demonstrate that the kinetic relation provides an efficient tool to handle 
complex fluids.

Recall that the design and the properties of difference schemes suitable for the numerical approximation of 
nonconservative systems \eqref{1.1} is very challenging. 
The main source of difficulty lies in the fact that shock waves to nonconservative systems are 
small-scale dependent and the dissipation terms induced by the numerical discretization tend to 
drive the propagation of the shocks. This phenomena was rigorously analyzed for scalar equations 
by Hou and LeFloch \cite{HouLeFloch}. On the other hand we emphasize that Glimm scheme and front-tracking algorithms 
do not contain any numerical dissipation and, actually, have been proven to converge to the correct 
solutions \cite{LeFloch-CPDE,LeFlochLiu,LeFloch-book}. The method
based on the kinetic relation proposed in the present paper
allows one to extend to nonconservative systems the conclusions made 
 for nonclassical shocks in \cite{HayesLeFloch1.5, LM} (and the references cited therein).

We begin with a general discussion of nonconservative hyperbolic systems arising in continuum physics 
in order to motivate our general approach 
proposed in the forthcoming section and developed on selected examples in the rest of this paper. 
The models of interest here naturally stand in a nonconservative form, and this is a direct consequence of 
simplifying assumptions 
which
are made in the derivation of these models; these assumptions are also necessary if a tractable model is to be found. 
Such assumptions typically originate in averaging procedures that intend to bypass the description 
of intricate mechanisms taking place at microscopic scales. 
The small scale fluctuations that are thought to be of lesser interest induce dissipative and/or relaxation phenomena at the macroscopic level, and can also be accounted for as source-terms. 

Most (if not all) nonconservative hyperbolic models arising in the applications
admit (several distinct) entropy balance laws 
which are consistent with the underlying dissipative and relaxation mechanisms. 
These additional balance equations, as we will show, provide a natural approach to formulating 
additional generalized jump conditions built from entropy rate productions.
Moreover, these entropy functions are sufficient in number to allow 
for a complete set of jump relations.

The objectives and results in this paper are as follows: 

First of all, as mentioned above, we restrict attention to a class of nonconservative systems
(defined in Section~\ref{RI-0} below)  
which encompasses, however, most of the models encountered in the applications. 
To motivate the proposed class of systems, 
we observe that,  in the applications we have in mind (e.g.~multi-phase and multi-fluid models): 
(1)   all but one of the equations \eqref{1.1} can be rewritten in a 
conservative form and, 
moreover, 
(2) the system \eqref{1.1} is endowed with a mathematical entropy, i.e.~a 
(strictly convex) nonlinear function $U=U(u)$ corresponding to an additional conservation law
satisfied by all smooth solutions. 

 For such systems, the concept of weak solutions introduced by Dal~Maso, LeFloch, and Murat \cite{DLM} 
can be simplified. Therein, a family of Lipschitz continuous paths was necessary to uniquely define 
the nonconservative product $A(u) \, \del_x u$ 
associated with the vector-valued field $u$. 
In contrast, for our particular class of nonconservative systems, one nonconservative product 
between scalar-valued functions, only, must be defined. 
This structure allows us to simply supplement the model \eqref{1.1} with an additional algebraic
scalar equation which, for each shock wave,
determines the {\sl entropy dissipation rate} associated with the entropy $U$. 
We call this additional jump condition a {\sl kinetic relation} 
and the entropy dissipation function a {\sl kinetic function.}  
In Section~\ref{RI-0} below, a precise definition is given. 
The main result of this section is a proof of the existence
of a solution to the Riemann problem for \eqref{1.1} which satisfies the prescribed kinetic relation. 
Our proof is a generalization of an argument given in Dal Maso, LeFloch and Murat \cite{DLM} 
in the setting of general families of paths.

It is remarkable that many models of interest arising in the applications 
take the form considered in Section~\ref{RI-0} below, and 
this will be illustrated in the following Section~\ref{PH-0}. 
In Section~\ref{MP-NS}, we focus on a model of particular importance, which arises in turbulent fluid dynamics. 
Taking into account the dissipation terms induced by the physical modeling, the existence and properties of 
associated traveling waves are established. In Section~ we characterize the right-hand states of traveling waves
with fixed left-hand states, which leads us to the desired kinetic relation. 
In turn, this provides us with the kinetic function 
needed to apply the general theory in Section~\ref{RI-0}.


\section{Nonconservative systems in fluid dynamics}
\label{PH-0}

To show the structure of the nonconservative systems of interest, 
it is worth to begin with the {\sl shallow water equations with topography} 
\be
\label{eq1sh}
\aligned
\del_t \rho + \del_x (\rho v) & = 0,
\\
\del_t (\rho v) + \del_x \big( \rho v^2 + g \, \frac{\rho ^2}{2} \big) - g \rho \, \del_x a & = 0,
\endaligned
\ee
where $\rho$ and $v$ are the mass density and the velocity of the fluid, respectively, 
and the (prescribed) topography function $a:\RR \to \RR_+$ depends on the spatial variable $x$
and is assumed to be solely piecewise Lipschitz continuous. Here, $g$ is the gravity constant. 
The product $g \rho \, \del_x a$ is nothing but a nonconservative product which is not defined in a classical sense
at points of discontinuity.  

By setting $\ubf := (\rho, \rho v)$, weak solutions should obey the following entropy inequality 
\be
\label{eqenerstar}
\aligned
& \del_t \Ucal(\ubf,a) + \del_x \Fcal(\ubf,a) \leq 0,
\\ 
& \Ucal(\ubf,a) := \rho E(v) + \rho a, \qquad e'(\rho)= \frac{p(\rho)}{\rho^2}, 
\\
& \Fcal(\ubf, a) := \rho \, \frac{v^3}{2} + \rho e(\rho) \, v + p(\rho)v + \rho v \, a.
\endaligned
\ee

Another model with a closely related structure is 
\be
\label{eq2s}
\aligned
\del_t (a\rho) + \del_x (a \rho v) =0,
\\
\del_t (a\rho v) + \del_x (a \rho v^2 + a \, p(\rho) ) - p(\rho) \, \del_x a =0,
\endaligned
\ee
which describes {\sl one-dimensional nozzle flows}  
as well as {\sl compressible flows in porous media.}  
Again, the function $a: \RR \to \RR$ is solely piecewise Lipschitz continuous,  
and denotes here the nozzle cross-section or the porosity function, respectively. 

By setting $\ubf := (a \rho, a\rho v)$, weak solutions to \eqref{eq2s} should obey the entropy inequality 
\be\label{eqentrop}
\aligned
& \del_t \Ucal(\ubf,a) + \del_x \Fcal(\ubf,a) \leq 0,
\\
& \Ucal(\ubf,a) = a^2 \rho \frac{v^2}{2} + a \rho \, e(\rho),
\\
&\Fcal(\ubf,a) = \big( \Ucal(\ubf,a) + p(\rho) \big) \, v.
\endaligned
\ee

The systems \eqref{eq1sh} and \eqref{eq2s} and closely related models with source-terms 
have received considerable attention over the past decade, 
from, both, analytical and numerical standpoints. We refer the reader to  \cite{LeFloch-IMA} (connection with the theory of nonconservative systems), 
\cite{GreenbergLeroux, GreenbergLeroux2, AudusseBouchutBristeauKleinPerthame, CLMP, Bouchut,GHS,Jin} 
(approximation by finite difference or finite volume schemes), 
and \cite{Liu82, Liu87, IsaacsonTemple, GoatinLeFloch, LT1,LT2} (construction of a Riemann solver). In particular, we refer the reader 
to Bouchut \cite{Bouchut} for a comprehensive review and to the references therein. 

We observe here that both models \eqref{eq1sh} and \eqref{eq2s} fall within the following class
of nonconservative hyperbolic models with singular source-term 
\be
\label{eq3s}
\del_t \ubf + \del_x \fbf(\ubf,a) - \gbf(\ubf,a)\del_x a=0, 
\ee
where $a$ is a given (piecewise Lipschitz continuous) function of the spatial variable $x$
and the unknown map $\ubf$ takes values in a convex and open domain $\Omega_{\ubf}\subset\RN$, 
while $\fbf: \Omega_{\ubf} \times \RR \to \RN$ and $\gbf:\Omega_{\ubf}\times\RR\to\RN$ are 
given smooth mappings.

Motivated by the structure of the above two examples, especially the entropy inequalities 
\eqref{eqenerstar} and \eqref{eqentrop} and in order to develop a general theory 
we assume that the hyperbolic system \eqref{eq3s} is endowed with 
a (sufficiently smooth) entropy function $\Ucal:\Omega_{\ubf} \times \RR \to \RR$ 
and a corresponding entropy flux $\Fcal:\Omega_{\ubf}\times\RR\to\RR$, 
so that solutions to \eqref{eq3s} satisfy the entropy inequality 
\be
\label{eq5s}
\del_t\Ucal(\ubf,a) + \del_x\Fcal(\ubf,a) \leq 0.
\ee
The principal examples of interest arising in the form \eqref{eq3s} in the applications do admit such an entropy.

The above class is known to include, after the seminal work by Greenberg and Leroux 
\cite{GreenbergLeroux}, the class of hyperbolic systems with source terms:
\be
\label{eq4s0}
\del_t \ubf+\del_x \fbf(\ubf) = \gbf(\ubf), 
\ee
which, by introducing the (rather trivial function) $a(x)=x$,  indeed take the form (cf.~\eqref{eq3s})
\be
\label{eq4s}
\del_t \ubf + \del_x \fbf(\ubf) - \gbf(\ubf) \, \del_x a=0. 
\ee
This nonconservative reformulation is useful for designing
``well-balanced schemes'', which properly account for the competition between the source term 
and the differential hyperbolic operator in the large time asymptotic $t \to +\infty$. 
(See \cite{GreenbergLeroux,GreenbergLeroux2,Bouchut}.)  
The (somewhat artificial but useful) system \eqref{eq4s}
admits a conserved entropy in the scalar case $n=1$, {\sl provided} the source $g$ does not vanish, 
namely it suffices to define 
$\Ucal'(u):=1/g(u)$ and $\Fcal'(u):= f'(u)/g(u)$.  


As advocated by LeFloch \cite{LeFloch-IMA} for the equations for nozzle flows \eqref{eq2s}, 
the prescribed function $a$, being independent of the time variable, can be regarded 
as an independent unknown of the following extended version of \eqref{eq3s}  
in the extended variable $u := (\ubf, a)$
\be
\label{eq5s1}
\aligned
\del_t \ubf + \del_x \fbf(\ubf,a) - \gbf(\ubf,a) \, \del_x a
& = 0, 
\\
\del_t a & = 0. 
\endaligned
\ee

Assuming from now on that the matrix $D_{\ubf}\fbf(\ubf,a)$ is diagonalizable for all $\ubf\in\Omega_{\ubf}$ 
and $a\in\RR$  
with real eigenvalues 
$\lambda_1(\ubf,a)$,\ldots, $\lambda_n(\ubf,a)$
and a full set of eigenvectors $r_1(\ubf,a)$,\ldots, $r_n(\ubf,a)$, 
it is obvious that \eqref{eq5s1} 
admits  
the same eigenvalues plus $0$ (with multiplicity $1$).
Moreover, it admits a full set of eigenvectors {\sl if and only if}
$\lambda_j(\ubf,a)\not=0$ for all $j=1,\ldots, n$. 
In general, \eqref{eq5s1} is solely weakly hyperbolic and, due to possible resonance effects, 
difficulties arise even in tackling the simplest initial value problem, i.e.~the
Riemann problem; see the pioneering work of Isaacson and Temple \cite{IsaacsonTemple}, 
as well as Goatin and LeFloch \cite{GoatinLeFloch} for a general Riemann solver. 
In the rest of this paper, our assumptions will explicitly exclude the resonance effect 
 in solutions under consideration.

While a rigorous definition of nonconservative products will wait until the following section, we can here 
already provide some preliminary discussion, based on an observation by LeFloch \cite{LeFloch-IMA}
for the nozzle flow equations and on the presentation in Bouchut \cite{Bouchut} for more general systems. 

With this non-resonance assumption enforced, we then observe from \eqref{eq5s1} that the variable $a$ is a 
Riemann invariant associated with the eigenvalue $\lambda_{n+1}(\ubf,a):= 0$. 
In other words, $a$ stays constant across waves associated with any other (non-vanishing) eigenvalue 
and, consequently,
the nonconservative product $\gbf(\ubf,a)\del_x a$ only needs to be defined for 
$(n+1)$-contact discontinuities. 

The entropy inequality \eqref{eq5s} should be satisfied as an equality in 
the sense of distributions across standing waves, hence  
\be
\label{eq6s}
\Fcal(\ubf_+,a_+)-\Fcal(\ubf_-,a_-) = 0.
\ee
From the physical viewpoint, we can further investigate the validity of \eqref{eq6s}, obtained
as a direct consequence of the augmented form \eqref{eq5s1}. To that purpose, we
specialize \eqref{eq6s} first to the case of the shallow water equations \eqref{eq1sh}
\be
\label{eq7s}
\frac{m^2}{2\rho_+^2}+e(\rho_+)+\frac{p(\rho_+)}{\rho_+} + a_+ =
\frac{m^2}{2\rho_-^2} + e(\rho_-)+\frac{p(\rho_-)}{\rho_-} + a_-,
\ee
where $m = : \rho_- v_- = \rho_+ v_+$ denote the mass, and second in the case of the nozzle flow equations \eqref{eq2s} 
\be\label{eq8s}
\frac{m^2}{2a_+^2\rho_+^2}+e(\rho_+)+\frac{p(\rho_+)}{\rho_+} =
\frac{m^2}{2a_+^2\rho_-^2} + e(\rho_-)+\frac{p(\rho_-)}{\rho_-},
\ee
in which $m := a_- \rho_- v_- = a_+ \rho_+ v_+$. 
The above equations can be implicitly solved in $\rho_+$ away from resonance (see \cite{Bouchut,GoatinLeFloch,LT1,LT2} 
for details) and stand at the very basis of the design of well-balanced schemes.

Furthermore, in concrete experiments with fluid flows, for instance in nozzles, 
it is observed that an abrupt change (modeled therefore by a discontinuity) 
in either the topography function,
the duct cross-section, or the porosity function, generally 
produces fine-scale features in solutions which may enter in competition with 
complex dissipation phenomena such as friction. 
To account for such dissipation mechanisms, the entropy law \eqref{eqenerstar} or \eqref{eqentrop}
cannot be any longer expressed as a conservation law across the standing wave. 
The associated entropy dissipation rates are the so-called ``singular
loss of momentum'' used by engineers and  
well-documented in the applied literature. It is necessary, 
on the ground of physical experiments, to replace \eqref{eq8s} by the more general condition 
\be
\label{eq9s}
\aligned
& (a\rho v)_+ - (a\rho v)_- = 0
\\
& \Jcal = - (a\rho v)_- \, \kappa(\ubf_-,a_-), 
\endaligned
\ee
with 
$$
\Jcal := (a\rho v)_+ \, \Big( 
\frac{v_+^2}{2} +e(\rho_+)+\frac{p(\rho_+)}{\rho_+} 
\Big) 
- 
(a\rho v)_- \, 
\Big(
\frac{v_-^2}{2} +e(\rho_-)+\frac{p(\rho_-)}{\rho_-}
\Big), 
$$
where the prescribed function $\kappa:\Omega_{\ubf} \times \RR \to
\RR_+$ defines mathematically the singular loss of momentum. 
The extension relative to \eqref{eq7s} is completely similar.


We continue this section with a more sophisticated model of compressible flows,
describing {\sl multi-fluid mixtures,}
where the variable $a$ introduced previously now stands for a fluid mass fraction. 
Following the celebrated review by Steward and Wendroff \cite{SW}, (stratified) multi-fluid models may be 
regarded as two distinct fluids evolving
with distinct velocities and distinct thermodynamic properties, 
each propagating within ``nozzles'' whose cross sections  
denoted by $a \in (0,1)$ and $1-a(x,t)$, respectively depend on the
spatial as well as the time variables (see \cite{warnecke} for instance). 
For notational convenience, $a$ is traditionally denoted by $\alpha_1$ (void fraction of the fluid $1$) 
and $(1-a)$ by $\alpha_2$ (void fraction of the fluid $2$),
with 
$$
\alpha_1(x,t) + \alpha_2(x,t) = 1.
$$

Furthermore, the evolution of $a$ is now described 
either via an algebraic closure equation (based on the isobaric assumption; see \cite{SW} for details) 
or
by considering it as an independent variable governed by a supplementary evolution equation.  
From the point of view of the present paper and in order to avoid instability issues due
to lack of hyperbolicity of the model, 
we adopt the second strategy,
following here Ransom and Hicks \cite{RH} and Baer and Nunziato \cite{BaerNunziato}. 
This approach was extensively investigated in recent years; see Gallouet, H\'erard, and Seguin \cite{GHS}, 
Berthon and Nkonga \cite{BerthonNkonga}, and the references therein. 
  
In turn, the multi-fluid model under consideration takes the form
\be
\label{bifluid}
\aligned
\del_t \alpha_1 + V_I(\ubf) \del_x\alpha_1 
& = \lambda(p_2-p_1),
\\
\del_t (\alpha_1\rho_1) + \del_x (\alpha_1\rho_1 \, u_1)  
& =0,
\\
\del_t (\alpha_1\rho_1 \, u_1) + \del_x(\alpha_1\rho_1 \, u_1^2 + \alpha_1 \, p_1) 
   - P_I(\ubf) \, \del_x\alpha_1 
& = 
        \lambda \, (u_2-u_1) + \eps \, \del_x(\mu_1 \, \del_x u_1),
        \\
\del_t (\alpha_2 \rho_2) + \del_x( \alpha_2\rho_2 \, u_2)  
& =0,
\\
\del_t (\alpha_2\rho_2 u_2) + \del_x(\alpha_2\rho_2 u_2^2 + \alpha_2 \, p_2) - P_I(\ubf) \, \del_x\alpha_2 
& = 
        - \lambda \, (u_2-u_1) + \eps \, \del_x(\mu_2 \, \del_x u_2),
\endaligned
\ee
where $\ubf := (\alpha_1,\alpha_1 \rho_1,\alpha_1\rho_1 u_1,\alpha_2 \rho2,\alpha_2\rho_2 u_2)$
is the vector-valued unknown. 
Here, 
the barotropic pressure laws $p_i=p_i(\rho_i)$ are
 assumed to satisfy the monotonicity condition $p_i'(\rho_i)>0$.
The relaxation parameter $\lambda>0$ may take arbitrarily large values, 
depending of the multi-fluid flow regime under consideration, 
while 
$\eps>0$ (the inverse of a Reynolds number) is usually small. 
Moreover, the (smooth) functions $V_I:\Omega_{\ubf}\to\RR$ and $P_I:\Omega_{\ubf}\to \RR$
represent the interfacial velocity and interfacial pressure, respectively. 
Following the original work by Ransom and Hicks \cite{RH}, one can set {\sl for instance}
\be
\label{choic}
V_I(\ubf) := \frac{1}{2}(u_1+u_2), 
\qquad
P_I(\ubf) := \frac{1}{2}(p_1+p_2). 
\ee

It turns out that, independently of the precise form of the constitutive equations,
the system \eqref{bifluid} admits five real eigenvalues, i.e. 
$$
V_I(\ubf), \quad u_i\pm c_i(\rho_i), 
$$
where $c_i^2(\rho_i) :=p'(\rho_i)>0$ (as well as a basis of right eigenvectors) 
{\sl if and only if}
\be
\label{rvectors}
|V_I(\ubf)-u_i|\not= c_i(\rho_i),\quad i=1,2.
\ee
In other words, like the (much simpler) model \eqref{eq5s1}, the principal (first-order) part of \eqref{bifluid} 
is only weakly hyperbolic if \eqref{rvectors} is violated. 
Here again, we tacitly assume that solutions under consideration do not develop
resonance phenomena.

One key constraint that arises in choosing the required closure laws for $V_I(\ubf)$ and $P_I(\ubf)$ 
is the existence of a mathematical entropy pair associated with \eqref{bifluid}. 
Interestingly, 
the total energy 
$$
\Ucal := \alpha_1\rho_1 E_1(\ubf) + \alpha_2\rho_2 E_2(\ubf)
$$
with $E_i(\ubf) :=\frac{u_i^2}{2}+e_i(\rho_i)$ is an entropy for \eqref{bifluid} if and only if
the interfacial closure laws $V_I(\ubf)$ and $P_I(\ubf)$ satisfy the 
{\sl interfacial compatibility condition}  
\be
\label{eqdeuxetoiles}
V_I(\ubf)(p_2-p_1) + P_I(\ubf) (u_2-u_1) = p_2 u_1 - p_1 u_2
\ee
for all states under consideration (see \cite{GHS} and \cite{cemracs} for instance). 
Indeed, under the assumption \eqref{eqdeuxetoiles},
smooth solutions of \eqref{bifluid} satisfy the entropy balance law 
\be
\label{en7} 
\aligned
& \del_t \Ucal(\ubf) + \del_x \Fcal(\ubf) = -\lambda(u_2-u_1)^2 - \lambda(p_2-p_1)^2 - \Dcal, 
\\
& \Ucal(\ubf) := (\alpha_1\rho_1 E_1(\ubf) + \alpha_2\rho_2 E_2(\ubf))
\\
& \Fcal(\ubf) := \Big( (\alpha_1\rho_1 E_1(\ubf)+\alpha_1 p_1)u_1 +
          (\alpha_2\rho_2 E_2(\ubf)+\alpha_2 p_2)u_2 \Big),
          \\
& \Dcal(\ubf) :=  \eps\mu_1(\del_x u_1)^2 + \eps\mu_2(\del_x u_2)^2
             - \eps \, \del_x(\mu_1\alpha_1\del_x u_1 + \mu_2\alpha_2\del_x u_2). 
\endaligned
\ee
The dissipation $\Dcal$ formally converges to a non-positive measure when $\eps \to 0$
and/or $\lambda \to +\infty$, so that in this limit we do have the entropy inequality 
\be
\label{ineq6}
\del_t \Ucal(\ubf) + \del_x \Fcal(\ubf) \leq 0. 
\ee


We conclude this section with another setting for complex
compressible materials which naturally gives rise to hyperbolic
equations with viscous perturbations in non conservation form. The
models under consideration may be regarded as natural extensions of the
usual Navier-Stokes equations. Such extensions make use of $N$
independent internal energies $(e_i)_{1\le i\le N}$ for governing $N$
independent pressure laws $(p_i(\tau,e_i))_{1\le i\le N}$. These PDE
models take the generic form:
\be
\label{s1_bis}
\aligned
\del_t \rho + \del_x(\rho u)
& = 0,
\\
\del_t(\rho u) + \del_x \Big( \rho u^2 + \sum_{i=1}^N p_i(\tau,e_i) \Big)
& =
    \eps \, \del_x \Big(\sum_{i=1}^{n}\mu_i(\tau,e_i)\del_x u \Big),
       \\
\del_t(\rho e_i) +\del_x(\rho u e_i) + p_i(\tau,e_i)\del_x u
& = 
   \eps \, \mu_i(\tau,e_i)(\del_x u)^2,
\endaligned
\ee
where $\rho>0$ denotes the density, $u\in{\mathbb R}$ is the velocity, and
$\tau=1/\rho>0$ is the specific volume. Here, $\eps>0$ stands the
inverse of the Reynolds number.

Several models from the Physics actually enter the proposed framework
and can be distinguished according to the precise definition of the
constitutive closure laws for the pressures and the
viscosities. Precise assumptions on the required state laws will be
addressed in Section~\ref{MP-NS} devoted to the analysis of the
traveling wave solutions of (\ref{s1_bis}).

Models from the plasma physics where the temperature of the electron
gases must be distinguished from the temperature of the other heavy
species, typically take the form (\ref{s1_bis}) with $N=2$ (see
\cite{CoquelMarmignon} for a presentation). Models from the physics of
compressible turbulent flows can also be seen to fall within the frame
of PDEs (\ref{s1_bis}). We refer the reader to
\cite{BerthonCoquel-2006,BerthonCoquel1,BerthonCoquel2,ChalonsCoquel-2006,ChalonsCoquel-2007}
for the mathematical and the numerical analysis of several models
ranging from two distinct internal energies, the so-called laminar and
turbulent ones, to $N>2$ different energies to account for a refined
description of the turbulent energy cascade. We also emphasize
that multi-fluid models, as those studied in \cite{BerthonNkonga},
enter the proposed framework.

In most if not all the applications to complex compressible materials,
the inverse of the Reynolds number $\eps$ modulating the strenght of
the viscous perturbation is a small parameter. Solutions of interest
therefore exhibit stiff zones of transitions, namely viscous shock
layers and boundary layers. Viscous shocks cannot be properly resolved
for mesh refinements of practical interest and we are thus led to
study the limit $\eps\to 0^+$ in the system (\ref{s1_bis}).

There exists several ways to tackle the system (\ref{s1_bis}) in the limit
$\eps\to 0^+$ depending on suitable change of variables. It can be
seen that (\ref{s1_bis}) can recast equivalently as:
\be\label{eq1_bis}
\Acal_0(\vw^\eps)\del_t \vw^\eps +\Acal_1(\vw^\eps)\del_x\vw^\eps = 
   \eps\del_x(\Dcal(\vw^\eps)\del_x \vw^\eps),
\ee
with $\Acal_0$ regular, or
\be\label{eq2_bis}
\del_t \ubf^\eps + \del_x \Fcal(\ubf^\eps) = 
   \eps \, \Rcal(\ubf^\eps,\del_x\ubf^\eps,\del_{xx}\ubf^\eps).
\ee
Namely in (\ref{eq1}), the diffusive operator writes in conservation
form while $\Acal_0(\vw)$ and $\Acal_1(\vw)$ are not Jacobian matrices
of some flux function. By contrast in (\ref{eq2_bis}), the first order
operator stands in conservation form but not the regularization
terms. 

The precise definitions of the change of unknown $\vw$ and $\ubf$ is addressed in
Section~\ref{MP-NS}. We just highlight at this stage that
concerning the equivalent form (\ref{eq1_bis}) and provided that suitable
estimates on the sequence of solutions $\vw_\eps$ hold true, the
right-hand side is expected to vanish in the limit $\eps\to 0^+$ in
the usual sense of the distributions. By contrast the left-hand side in non
conservation form may be handled thanks to the theory of family of
paths introduced by LeFloch \cite{LeFloch-IMA}, Dal Maso, LeFloch and
Murat \cite{DLM}. As far as the next equivalent form (\ref{eq2_bis} is
considered, the left-hand side now stands in conservation form and can
be treated in the usual sense of the distributions. In opposition, the
right-hand side cannot any longer be expected to converge to $0$,
generally speaking, but merely to a bounded Borel measure concentrated
on the shocks of the limit solutions. The next section provides a
convenient framework for handling the required passage to the limit
in the PDEs (\ref{eq2_bis}).


\section{Kinetic relations for nonconservative systems}
\label{RI-0}

Having in mind the examples described in the previous section, we present one of 
the main contributions of the present paper, i.e.~the {\sl concept of kinetic relations for nonconservative systems,}
 which 
allows us to 
rigorously define certain nonconservative products arising in the applications.
 
Recall that weak solutions to nonconservative systems are defined in the class of
functions with bounded 
variation (BV). By standard regularity theorems, such functions can be handled essentially 
as if they were piecewise Lipschitz continuous. Throughout the present paper and for simplicity in the presentation, 
we restrict attention to piecewise Lipschitz continuous functions and refer to \cite{DLM} for details
of the the DLM theory. 

For simplicity in the presentation, we restrict attention to solutions defined in a neighborhood of a constant state in $\RN$
which can be normalized to be the origin. 
We denote by $\Bzero$ the ball centered at the origin and of small radius $\delta_0>0$. 
Dal~Maso, LeFloch and Murat's definition is based on prescribing a  
{\sl family of Lipschitz continuous paths} $\phi=\phi(s;u_0, u_1) \in \Bzero$ ($s \in [0,1]$),  
which allows one to connect any two points $u_0, u_1$ in $\Bone$ for some $\delta_1 \leq \delta_0$. 
In particular, it is assumed that 
\be 
\phi(0; u_0, u_1) = u_0, \qquad  \phi(1; u_0, u_1) = u_1. 
\label{2.0b} 
\ee
(See \cite{DLM,LeFloch-JHDE} for the precise conditions, omitted in this short review.) 
As proposed in LeFloch \cite{LeFloch-IMA}, this family of paths should be determined from 
traveling wave solutions of an augmented model.

Indeed, it has been recognized that weak solutions $u$ of \eqref{1.1} depend on the effect of small scales 
that have been neglected at the hyperbolic level of modeling, but are taken into account in 
the augmented version 
\be
\del_t u + A(u) \, \del_x u = R(u, \eps \, u_x, \eps^2 u_{xx}, \cdots), 
\label{2.0}
\ee 
where $R^\eps =0$ if $\eps= 0$. The family of paths determined by traveling wave trajectories 
precisely yields the ``missing information'' required to set-up the hyperbolic theory. 

A (piecewise Lipschitz continuous) function $u=u(x,t)$ is called a {\sl weak solution} of 
the nonconservative system \eqref{1.1} 
if $u$ satisfies the equations \eqref{1.1} in a classical sense in the regions where it is Lipschitz continuous and, 
additionally, the following generalization of the Rankine-Hugoniot jump relation 
holds along every curve of discontinuity of $u$. Precisely, 
for any shock wave connecting two states $u_0, u_1$ at the speed $\Lamb= \Lamb(u_0,u_1)$, 
\be
u(x,t) = \begin{cases}
u_0,       &  x< \Lamb \, t, 
\\
u_1,       &  x> \Lamb \, t, 
\end{cases} 
\label{2.1} 
\ee
we impose the {\sl generalized jump relation}
\be
- \Lamb \, (u_1 - u_0) 
+ \int_0^1 A(\phi(s; u_0,u_1)) \, \del_s \phi(s; u_0,u_1) \, ds = 0. 
\label{2.2}
\ee
Note that, in the conservative case when $A(u) = Df(u)$ for some flux-function $f$, this relation reduces to 
$$
- \Lamb \, (u_1 - u_0) + f(u_1) - f(u_0) =0, 
$$
which is {\sl independent} of the paths $\phi$ and is nothing but the standard jump relation.  

Based on the above definition, one can solve  \cite{DLM} the Riemann problem for \eqref{1.1}, 
corresponding to the piecewise constant initial data 
\be 
u(x,0) = \begin{cases}
u_l,    &  x<0, 
\\
u_r,   &  x>0, 
\end{cases}
\label{2.3} 
\ee
where $u_l, u_r$ are  constants in $\Btwo$ with $\delta_2 \leq \delta_1$. This construction generalizes Lax's standard construction  
for conservative systems \cite{Lax1,Lax2}. 
Recall that (admissible) shock waves must be constraint by Lax shock inequalities 
(for some $j=1, \ldots, N$) 
\be
\lam_j(u_0) > \Lamb > \lam_j(u_1). 
\label{2.Lax}
\ee
 
The Riemann solver can then be used to design numerical schemes for the approximation 
of the general initial value problem, e.g.~Glimm or front tracking schemes. 

In certain applications, it has been found convenient to avoid introducing the whole family of paths $\phi$. 
It is precisely our
purpose in the present paper to introduce, for a particular class of nonconservative systems, 
a new definition of weak solutions, 
which imposes Rankine-Hugoniot jump relations
in the form of ``kinetic relations'' and 
does not require the knowledge of 
any ``internal structure'' for shock waves. 


We will assume that the nonconservative system under consideration {\sl formally} 
admits $N$ conservation laws, so we consider the system  
\be
\del_t u + \del_x f(u) = 0,  
\label{2.4} 
\ee
which consists of conservation laws valid for {\sl smooth} solutions, only.   
Our goal is to describe singular limits \eqref{2.0},   
where $R=R^\eps$  
is a {\sl nonconservative regularization.} 
Precisely, we are going to supplement \eqref{2.4} with $N$ jump relations, 
referred to as ``kinetic relation'', which determines the dynamics of shocks in weak solutions to \eqref{2.4}.  

We suppose that in an open and convex domain $\Ucal \subset \RN$ of the phase space, the system \eqref{1.1} is strictly hyperbolic, 
with eigenvalues $\lam_1(u) < \ldots < \lam_N(u)$
and basis of eigenvectors $l_i(u), r_i(u)$. 
Let $\Lcal \subset \RR$ be a compact set containing all speeds under consideration in the problem.

\begin{definition}
\label{A-2.1}  
A {\sl kinetic function} is a Lipschitz continuous map $\Phi:  \Ucal \times \Lcal \to \RN$
satisfying (for $j=1, \ldots, N$) 
\be 
\aligned 
& \Phi(u, \lam_j(u)) =0,  \qquad u \in \Ucal,
\\
& | l_j(u) \cdot \del_{\Lam} \, \Phi(u, \Lam) | 
         \leq  c_1 \, |\Lam - \lam_j(u)|, \qquad (u, \Lam) \in \Ucal \times \Lcal,   
\endaligned 
\label{2.6} 
\ee 
for some constant $c_1>0$. 
Given a a kinetic function $\Phi$, a piecewise Lipschitz solution $u=u(x,t) \in \Ucal$ 
is called a {\sl $\Phi$-admissible weak solution}   to \eqref{1.1} 
if the differential equations \eqref{2.4} are satisfied in each region of continuity of $u$ 
and moreover, 
along any curve of discontinuity of $u$, connecting some values $u_-,u_+$ at the speed $\Lam$, 
the following {\sl kinetic relation} holds 
\be
- \Lam \, (u_+ - u_-) + f(u_+) - f(u_-)  = \Phi(u_-, \Lam). 
\label{2.5} 
\ee
\end{definition}  
 
In certain applications, it may be more convenient to express the kinetic functions in terms of the left- and right-hand states, 
that is, $\Phi=\Phi(u_-,u_+)$. In the applications, the kinetic function $\Phi$ should be determined from traveling wave solutions
of a specific system \eqref{2.0} and should be thought of as a ``correction'' to the standard Rankine-Hugoniot 
relation. 

By introducing 
the Borel measure denoted by $\mu_u^\Phi$ that vanishes in the regions of continuity of $u$ 
and has the mass $\Phi(u_-, \Lam)$ along its curves of discontinuity, 
we easily see that 
Definition~\ref{A-2.1} is equivalent to the requirement (see \cite{LeFloch-ARMA}) 
\be 
\del_t u + \del_x f(u) =  \mu_u^\Phi,  
\label{2.5-b}
\ee
which is regarded as an equality between bounded measures. 
Note that we recover the usual conservative case by simply choosing both $\Phi$ and $\mu_u^\Phi$ to vanish identically. 
In general, $\mu_u^\Phi$ depends strongly on the function $u$. 

In the rest of this section we study the case of genuinely nonlinear systems. This
 assumption allows us to use the shock speed as a regular parameter along the (generalized) Hugoniot curve.

\begin{theorem} [Riemann problem for nonconservative systems with kinetic relations]
\label{A-2.2} 
Suppose that \eqref{2.4} is a strictly hyperbolic system in a neighborhood $\Bcal_{\delta_0}$ of 
the origin $0$ and admits genuinely nonlinear characteristic fields only, i.e. 
$$
(\nabla\lam_j \cdot r_j)(0) > 0, \qquad  j=1, \cdots, N. 
$$
Let $\Phi =\Phi(u,\Lam)$ be a (Lipschitz continuous) kinetic function defined in the neighborhood $\Bzero \times \Lcal$   
for some sufficiently small $\delta>0$ by 
$$
\Lcal := \bigcup_j \Lcal_j, 
\qquad 
\Lcal_j := (\lam_j(0) - \delta, \lam_j(0) + \delta).  
$$
Then, there exists $\delta_1 \leq \delta_0$ such that 
the Riemann problem \eqref{2.3}, \eqref{2.4} with data $u_l, u_r \in \Bone$
admits a unique $\Phi$-admissible weak solution in the class of piecewise smooth solutions 
consisting of a combination of rarefaction waves and shock waves satisfying the kinetic relation. 
Moreover, the corresponding wave curves are solely Lipschitz continuous.  
\end{theorem}  

Clearly, under the assumptions of the above theorem, \eqref{2.6} implies that, for $j=1, \ldots, N$, 
\be
|l_j(u) \cdot \Phi(u, \Lam)| \leq c_2 \, |\Lam - \lam_j(u)|^2, \qquad (\lam, u) \in \Ucal \times \Lcal_j 
\label{2.kinetic}
\ee
for some $c_2>0$. 

\begin{proof} We want to generalize the proof given in \cite{DLM} for general families of paths; 
see also the related proof in \cite{HayesLeFloch2} for nonclassical shocks.
We are going to show that the given set of jump conditions \eqref{2.5} suffices to determine a (generalized) 
Hugoniot curve uniquely, and we will investigate whether its tangency and regularity 
properties. The rest of the proof (selection of the admissible part of the Hugoniot curve, 
actual construction of the wave curves, Riemann solution) then follows as in \cite{DLM} and will be omitted.   

We denote by $\lamb_j(u_0,u_1)$ and $\lb_j(u_0, u_1)$ the eigenvalues and left-eigenvectors of the 
averaged matrix 
$$
A(u_0, u_1) := \int_0^1 Df(u_0 + m \, (u_1 - u_0)) \, dm. 
$$
In a neighborhood of the point $(u_0, \lam_j(u_0))$, we consider the kinetic relation 
\be
G(\Lam, u_1) := - \Lam \, (u_1 - u_0) + f(u_1) - f(u_0) 
- \Phi(u_0, \Lam) =0. 
\label{2.7} 
\ee
Fix some index $i$ and let us restrict attention to the (nonlinear) cone-like region $K$ determined by the two conditions 
on $u_1 \in \Bcal_{\delta_1}$ 
$$
\aligned 
& \bigl|(u_1 - u_0) \cdot l_i(u_0, u_1)\bigr| \geq C_*  \, |\Lam - \lam_i(u_0)|, 
\\
& |u_1 - u_0|  + |\Lam - \lam_i(u_0)| < \delta_2, 
\endaligned 
$$
where a condition on $C_*>0$ will be imposed below. Observe that $G(u_0, \lam_i(u_0)) = 0$.

Multiplying the generalized jump relation \eqref{2.7}  by $\lb_i(u_0, u_1)$ we find 
$$
\aligned 
0 
& = \lb_i(u_0, u_1)\cdot \bigl(A(u_0,u_1) - \Lam \bigr) \, (u_1 - u_0) 
      - \lb_i(u_0, u_1) \cdot \Phi(u_0, \Lam) 
\\ 
& =  \bigl(\lamb_i(u_0, u_1) - \Lam \bigr)  \, \lb_i(u_0, u_1)\cdot (u_1 - u_0) 
       - \lb_i(u_0, u_1) \cdot \Phi(u_0, \Lam). 
\endaligned 
$$
Therefore, we can express the shock speed $\Lam= \Lamb(u_0, u_1)$ in the form 
\be
\label{her}
0 = \Lamb - \lamb_i(u_0, u_1) + {\lb_i(u_0, u_1)\cdot \Phi(u_0, \Lamb)
                 \over  \lb_i(u_0, u_1)\cdot (u_1 - u_0)} =: \Omega(u_1, \Lamb). 
\ee

Now, observe that the function $\Omega$ satisfies 
$$
\aligned 
\Big| {\del \Omega \over \del\lam} (u_1, \Lamb) - 1 \Big|
& = {\lb_i(u_0, u_1)\cdot \del_\lam \Phi(u_0, \Lamb)  \over \lb_i(u_0, u_1) \cdot (u_1 - u_0)}
\\
& \geq - c_2 \, O(1) \, {|u_1 - u_0 | + |\Lamb - \lamb_i(u_0)|  \over | \lb_i(u_0, u_1) \cdot (u_1 - u_0)|}, 
\endaligned 
$$ 
where the constant $O(1)$ depends only on the flux. Hence, we have 
$$
{\del \Omega \over \del\Lam} (u_1, \Lamb)
=
1 + {c_2 \over C_*} \, O(1),
$$ 
which is positive provided $c_1$ is sufficiently small. 
As a consequence, the implicit function for Lipschitz continuous mappings 
applies and shows that the implicit equation \eqref{her} determine the shock speed $\Lamb= \Lamb(u_0,u_1)$
uniquely.

Next, we consider the remaining components, corresponding to $j \neq i$: 
$$
H(u_0, u_1) :=  l_j(u_0, u_1)\cdot (u_1 - u_0) 
         - {l_j(u_0, u_1)\cdot \Phi(u_0, \Lamb) \over \Lamb(u_0,u_1) -  \lam_j(u_0,u_1)}. 
$$
Denoting by $L(u_0)$ the $N \times (N-1)$ matrix of vectors $l_j(u_0)$ for $j \neq i$,  
we can compute the differential of $H$, as follows: 
$$
\aligned 
{DH \over Du_1} (u_0, u_1) 
=  L(u_0) 
&+ O(1) \, |u_1 - u_0| + O(1) C_1 {|\Lamb(u_0,u_1) -  \lam_i(u_0)|^2 \over |\Lamb(u_0,u_1) -  \lam_j(u_0)| }
\\
& + O(1)\,  C_1 |\Lamb(u_0,u_1) -  \lam_i(u_0)| \, \left|{\del\Lamb \over \del u_1}(u_0,u_1)\right| 
\\  
& +  O(1) \, C_1  |\Lamb(u_0,u_1) -  \lam_i(u_0)|^2 \, \left|{\del\Lamb \over \del u_1}(u_0,u_1)  
- {1 \over 2} \, \nabla \lam_i(u_0)\right|, 
\endaligned 
$$ 
where we used that $\Lamb(u_0,u_1) -  \lam_j(u_0,u_1)$ is bounded away from $0$. 
Hence, we find  
$$
{\del H \over \del u_1} (u_0,u_1) = L(u_0) + o(1) + o(1)\, \left|{\del\Lamb \over \del u_1}(u_0,u_1)\right|.  
$$  

On the other hand,  the $u_1$-derivative of the shock speed satisfies 
$$
\aligned 
& {\del \Lamb \over \del u_1}(u_0,u_1)  
\\
& = {1 \over 2} \, \nabla \lam_i(u_0) + O(1) {|\lam -  \lam_i(u_0)|^{b+1} \over |l_i(u_0,u_1) \cdot (u_1- u_0)| }
\\
& \quad  +  O(1) \, {|\Lamb -  \lam_i(u_0)|^2 \over |l_i(u_0,u_1) \cdot (u_1- u_0)|^2 } 
      + O(1) {|\Lamb -  \lam_i(u_0)| \over |l_i(u_0,u_1) \cdot (u_1- u_0)| } \, 
{\del \Lamb \over \del u_1}(u_0,u_1),   
\endaligned 
$$
which shows that 
$$
{\del \Lamb \over \del u_1}(u_0,u_1) = {1 \over 2} \, \nabla \lam_i(u_0) + o(1). 
$$ 

In conclusion, ${\del H \over \del u_1} (u_0,u_1) = L(u_0) + o(1)$, and the implicit function 
theorem applies to the set of equations $H(u_0,u_1)=0$,
 which therefore  
determines a unique shock curve $s \mapsto u_1 = u_1(s;u_0)$, defined locally near $u_0$. 
Near the base point $u(0) = u_0$, 
the tangent of this curve is defined almost everywhere 
and, due to the smallness of the constant $c_1$ in \eqref{2.6}, 
takes its values in a small neighborhood of the eigenvector $r_i(u_0)$. 
\end{proof}


We now introduce a class of nonconservative system to which the framework in the previous
subsection can be applied. 
We assume that  
the first $N-p$ equations in \eqref{1.1} take a conservative form 
while the remaining $p$ equations are nonconservative. In other word, we set $u=(v,w)$ and we consider the nonconservative systems 
\be
\aligned 
& \del_t v + \del_x g(v,w) = 0, \\
& \del_t w + B(v,w)\, \del_x v + C(v,w) \, \del_x w =0.  
\endaligned 
\label{2.6bb}
\ee
Here $g=g(v,w) \in \RR^{N-p}$ while $B=B(v,w), C=C(v,w)$ are $p\times (N-p)$ and 
$p \times p$ matrix-valued mappings, respectively. 

It must be stressed that the assumption made here refers directly 
to the set of equations listed in \eqref{1.1} or to {\sl linear} combinations
of them. Of course, nonlinear functions of the original variable $u$  
cannot be considered at this level of the analysis, in general,
since discontinuous solutions are sought. 

Our second assumption is the existence of $p$ mathematical entropy 
pairs. That is, we assume that there exist $k$ strictly convex 
functions $U_k=U_k(v,w)$ together with their associated flux $F_k=F_k(v,w)$ such that 
\be
\del_t U_k(v,w) + \del_x F_k(v,w) =0, \qquad k=1, \cdots, p,  
\label{2.7b} 
\ee
holds for all {\sl smooth} solutions to \eqref{2.6}. We search for solutions satisfying the {\sl entropy inequality} 
\be
\del_t U_k(v,w) + \del_x F_k(v,w) \leq 0,   \qquad k=1, \cdots, p. 
\label{2.8}
\ee
Many of the models of interest take the form \eqref{2.6}--\eqref{2.8}. 

\begin{definition}
Nonlinear hyperbolic systems in nonconservative form that have the structure \eqref{2.6bb}, 
admit 
at least $p$ mathematical entropies, and satisfy the non-degeneracy condition
\be
det\left(\nabla_w U_1(v,w), \cdots, \nabla_w U_p(v,w)\right) 
\neq 0
\label{2.12}
\ee
are called {\rm nonconservative systems endowed with a full set of entropies.} 
\end{definition}


We now focus on the entropy dissipation associated with the entropies  
$U_k$. The basic idea is to replace the nonconservative
equations in the system \eqref{2.6} with conservative 
equations for the entropy dissipation but the latter involving a measure
source-term. Of course it is necessary for $(v,w) \mapsto (v, U(v,w))$ to define a change of variable, say $U_{w} \neq 0$. 

Observe first that the inequality \eqref{2.8} implies a 
constrain on shock waves, i.e., with the notation introduced earlier in \eqref{2.1}, 
\be
\aligned 
E_k(\Lam; u_0, u_1) 
& := -\Lam \, \bigl(U_k(u_1) - U_k(u_0)\bigr) + F_k(u_1) - F_k(u_0) 
\\
& \leq 0,
\endaligned  
\label{2.9}
\ee
for all $k=1, \ldots p$. 
On the other hand, the first $N-p$ equations in (2.6) yield $N-p$ jump relations
in the fully explicit form 
\be
-\Lam \, (v_1 - v_0) + g(v_1, w_1) - g(v_0, w_0) = 0. 
\label{2.10}
\ee
Since $p$ jump relations are ``missing'', we supply it in the form 
\be
E_i(\Lam; u_0, u_1) = \Phi_i(\Lam; u_0) \leq 0,  \qquad i=1, \cdots p, 
\label{2.11}
\ee
which we refer to as a kinetic relation and 
where $\Phi$ is a given ``constitutive'' function, called a ``kinetic function'', to be determined case by case in the examples.

\begin{definition}
\label{A-2.3} Let $\Phi=(0, \ldots, 0, \Phi_1, \ldots \Phi_p)$ be a kinetic function.  
A piecewise Lipschitz continuous function $u=(v,w)$ is called a 
{\rm $\Phi$-admissible solution of the nonconservative system} 
\eqref{2.6} if it satisfies the equations in a classical sense in the 
regions of continuity and if each propagating discontinuity 
satisfies the $N-p$ jump relations \eqref{2.10} together with 
the kinetic relations \eqref{2.11}. 
\end{definition}

We reformulate the main result, in a slightly weaker form which is adapted to the present context, since 
it is natural to assume that the entropy dissipation is of cubic order near the base point.  

\begin{corollary}[Riemann problem for nonconservative systems endowed with a full set of entropies] 
\label{A-2.4} 
Consider a nonconservative system endowed with a full set of entropies.  
Suppose that the system is strictly hyperbolic and genuinely nonlinear in the neighborhood of some state $u_*= (v_*,w_*)$. 
Let $\Phi_i =\Phi_i(u_0, u_1)$ be a regular function 
defined in the neighborhood of each speed $\lam_j(u_*)$ 
for $j=1, \cdots, N$ and satisfying for all $u_0, u_1$ 
\be
\aligned
& \Phi_i(u_0, \lam_i(u_0)) = 0,  
\\
& \del_{\Lam} \Phi_i(u_0, \Lam) = O(1) \, (\Lam - \lam_j(u_0))^2,  
\endaligned 
\label{2.13}
\ee
where $O(1)$ denotes a positive and bounded function. 
Then, the corresponding Riemann problem admits a unique admissible solution 
in the class of piecewise smooth solutions consisting of 
a combination of rarefaction waves and admissible shock waves. 
\end{corollary}

The theory of the present section applies to the examples listed in Section~2, at least when
resonance effect is avoided. It is straighforward to include linearly degenerate characteristic fields provided
the kinetic function is chosen to vanish identically for those fields. 


\section{Multi-pressure Navier-Stokes system}
\label{MP-NS} 

In this section, we establish the existence and uniqueness of the traveling wave
solutions of the multi-pressure Navier-Stokes equations introduced in
Section~2, under fairly general assumptions on the
pressure and viscosity closure laws. The equations under consideration
were stated in \eqref{s1_bis}.  
Each smooth pressure law $p_i(\tau,e_i)$, $1\le i\le N$, is assumed
to obey the second principle of the thermodynamics, namely
\be
\label{thermo}
T_i(\tau,e_i) \, ds_i = de_i + p_i(\tau,e_i) \, d\tau,
\ee
where $T_i(\tau,e_i)>0$ is the corresponding temperature variable
and $s_i>0$ denotes the specific entropy. 

The map $(\tau,s_i)\mapsto e_i(\tau,s_i)$ is thus well-defined and is assumed to be strictly convex. 
In addition, the following asymptotic conditions are assumed 
\be
\label{thermo1}
\aligned
\lim_{\tau\to 0^+} e_i(\tau,s_i) = +\infty,
\qquad
\lim_{s_i\to +\infty} e_i(\tau,s_i) = +\infty,
\qquad
\lim_{\tau\to +\infty} e_i(\tau,s_i) = 0.
\endaligned
\ee
It follows that 
\be
\label{Tpositif}
p_i(\tau,s_i)=-\frac{\del e_i}{\del \tau}(\tau,s_i) >0, 
\quad
\qquad 
T_i(\tau,s_i)=\frac{\del e_i}{\del s_i}(\tau,s_i) >0.
\ee

Furthermore, the following assumptions are introduced for any given $\tau>0$:
\begin{eqnarray}
&&\label{hypH2}
  \frac{\del p_i}{\del s_i}(\tau,s_i)>0,
\\
&&\label{hypH1}
  \sum_{i=1}^{N} \frac{\del^2 p_i}{\del \tau^2}(\tau,\vs)>0,\quad
\\
&&\label{hypH3}
  \lim_{\tau\to0^+}  \sum_{i=1}^{N}p_i(\tau,\vs)=+\infty,\quad
  \lim_{\tau\to+\infty}  \sum_{i=1}^{N}p_i(\tau,\vs)=0,
\\
&&\label{hypH4}
  \lim_{\tau\to0^+} \sum_{i=1}^{N} \frac{\del  p_i}{\del
     \tau}(\tau,\vs)=-\infty,\quad
  \lim_{\tau\to+\infty} \sum_{i=1}^{N} \frac{\del  p_i}{\del
  \tau}(\tau,\vs)=0,
\\
&&\label{hypH5}
  \lim_{s_i\to+\infty} \frac{\del p_i}{\del \tau}(\tau,s_i)=-\infty,
\end{eqnarray}
where $\vs=(s_1,...,s_N)$. We refer to Menikoff and Plohr \cite{MP} for general properties of the fluid equations
and the equation of state. 
Next, the viscosity laws are given smooth functions with 
\be
\label{condvisco}
\mu_i(\tau,s_i)\ge0, \quad 1\le i\le N,
\qquad\mbox{and}\qquad
\mu(\tau,s):=\sum_{i=1}^{N}
\mu_i(\tau,s_i)>0.
\ee
To shorten the notation, the PDE's system
(\ref{s1_bis}) is given in the condensed form
$$
\del_t \vv^\eps + \Acal(\vv^\eps)\del_x\vv^\eps = \eps\Bcal(\vv^\eps,\del_x\vv^\eps,\del_{xx}\vv^\eps),
$$
with $\vv$ in the phase space:
$$
\Omega_{\vv} = \{
\vv=(\rho,\rho u,(\rho e_i)_{1\le i\le N})\in\RR^{N+2}; \rho>0,\rho
u\in\RR,\rho e_i>0, 1\le i\le N
\}.
$$
The basic properties of (\ref{s1_bis}) are summarized in the following statement.

\begin{lemma}
The underlying first-order part from (\ref{s1_bis}) is hyperbolic in
$\Omega_{\vv}$ and admits three distinct eigenvalues
\begin{eqnarray} \quad \qquad 
\lambda_1(\vv) = u-c(\vv),~
\lambda_2(\vv) = ... = \lambda_{N+1}(\vv) = u,~
\lambda_{N+2}(\vv) = u-c(\vv),\\
\end{eqnarray}
where we set
\begin{eqnarray}
c^2(\vv)=\sum_{i=1}^N -\tau^2\frac{\del p_i}{\del \tau}(\tau,s_i).
\end{eqnarray}
The extreme fields are genuinely nonlinear while the intermediate ones
are linearly degenerate. Then, smooth solutions of (\ref{s1_bis}) satisfy
the additional conservation law 
\be
\del_t (\rho E)^{\eps} + 
  \del_x \left( \{\rho E\}(\vv^\eps)+\sum_{i=1}^N p_i(\tau^\eps,s_i^\eps)
    u^\eps\right) = \nonumber\\ 
  \eps\del_x \left( \sum_{i=1}^N \mu_i(\tau^\eps,s_i^\eps)
    u^\eps\del_x u^\eps \right),
\ee
where the total energy reads
\be\label{totalE}
(\rho E) = \frac{(\rho u)^2}{2\rho} + \sum_{i=1}^N \rho e_i.
\ee
At last, the smooth solutions of (\ref{s1_bis}) obey the $N$ balance
equations 
\be
\del_t (\rho s_i)^\eps + \del_x \left( (\rho s_i)^\eps u^\eps \right)= 
\eps \frac{\mu_i(\tau^\eps,s_i^\eps)}{T_i(\tau^\eps,s_i^\eps)} (\del_x u^\eps)^2. 
\ee
\end{lemma}

As already claimed, changes of variables with distinctive features
allow to recast (\ref{s1_bis}) either in the equivalent form:
\be\label{eq1}
\Acal_0(\vw^\eps)\del_t \vw^\eps +\Acal_1(\vw^\eps)\del_x\vw^\eps = 
   \eps\del_x(\Dcal(\vw^\eps)\del_x \vw^\eps),
\ee
with $\Acal_0$ regular, or in the form 
\be\label{eq2}
\del_t \ubf^\eps + \del_x \Fcal(\ubf^\eps) = 
   \eps \, \Rcal(\ubf^\eps,\del_x\ubf^\eps,\del_{xx}\ubf^\eps).
\ee
We briefly discuss the changes of variables involved in
(\ref{eq1}) and (\ref{eq2}). Concerning (\ref{eq1}), we first observe
that summing the $N$ governing equations for the internal energies
yields 
$$
\del_t\rho e+\del_x\rho e u+\sum_{i=1}^N p_i\del_x u =
  \eps \mu(\tau,\vs) (\del_x u)^2,
$$
so that the following identities are easily checked:
\begin{eqnarray}
&&
\mu(\tau,\vs) \Big( 
    \del_t \rho e_i + \del_x(\rho e_i u) +p_i(\tau,s_i)\del_x u
\Big) - \nonumber\\
&& \hspace*{0.5cm}
\mu_i(\tau,s_i) \Big(
    \del_t\rho e + \del_x(\rho e u) +\sum_{i=1}^N p_i\del_x u
\Big) =0,
\quad 1\le i\le N-1.
\label{eq3}
\end{eqnarray}
Since $\rho e =\rho E-(\rho u)^2/(2\rho)$, the conservation laws for
$\rho$, $\rho u$ and $\rho E$ supplemented by the $(N-1)$ balance
equations in (\ref{eq3}) can be seen to give the equivalent form
stated in (\ref{eq1}) when defining $\vw=(\rho,\rho,\rho E,
(\rho e_i)_{1\le i\le N-1})$. A direct calculation shows that 
${\rm det}\Acal_0(\vw)=\mu(\tau,\vs)^{n-1}>0$.

Concerning system (\ref{eq2}), several changes of variables can be used and we
advocate in the sequel the change of variables $\vv\in\Omega_{\vv} \mapsto
\ubf(\vv)\in\Omega_{\ubf}$, with $\ubf(\vv)=(\rho,\rho u,(\rho s_i)_{1\le
  i\le N})$.

As underlined in Section~\ref{RI-0}, both approaches rely on the
study of the traveling wave solutions of (\ref{s1_bis}). Due to the frame
invariance properties satisfied by the PDE model (\ref{s1_bis}), it
suffices to analyze traveling waves solutions associated with the
first extreme field. With this respect, the main result of this section is
as follows.

\begin{theorem}[Traveling wave solutions to the multi-pressure
    Navier-Stokes system]
\label{theo_main}
Consider the multi-pressure Navier-Stokes system \eqref{s1_bis} when the pressure satisfies the 
positivity, convexity, and asymptotic conditions
 \eqref{thermo1}--\eqref{condvisco}.
Let $\ubf_L \in \Omega_{\ubf}$ and $\sigma\in\RR$ be given such that
\be\label{lax}
\frac{u_L-\sigma}{c(\ubf_L)}>1,\quad
c^2(\ubf_L) = \sum_{i=1}^N -\tau^2_L 
    \frac{\del p}{\del \tau}(\tau_L,(s_i)_L).
\ee
Then, there exists a unique traveling wave solution to \eqref{s1_bis}
issuing from the left-hand state $\ubf_L$ and reaching some right-hand
state $\ubf_R\in\Omega_{\ubf}$ with 
\be\label{lax_bis}
0<\frac{u_R-\sigma}{c(\ubf_R)}<1
\ee
\end{theorem}

The proof of this result will follow from the characterization of a
positively invariant compact set of $\Omega_{\ubf}$. Then the Lasalle
invariance principle applied in connection with a suitable Lyapunov
function ensures the existence of a traveling wave. Uniqueness is
obtained as a simple consequence of the center manifold theorem.

We gather here some of the notation used repeatedly hereafter and
give the precise form of the autonomous system which governs the
viscous profiles we study for existence. Simple but useful
geometrical properties induced by the corresponding vector field will be
then put forward.

Due to Galilean invariance, it suffices to consider the case of a null
velocity $\sigma$. The precise form of the PDE system governing the
traveling wave solutions then follows when restricting attention to
solutions which depend solely on $x$:
\be\label{sysG1}
\aligned
(\rho u)_x&=0,\\
(\rho u^2+p(\tau,\vs))_x&=(\mu(\tau,\vs)u_x)_x,\\
T_i(\tau,s_i)(\rho s_i u)_x&=\mu_i(\tau,\vs)(u_x)^2,\quad 1\le i\le N.
\endaligned
\ee
The first equation in (\ref{sysG1}) implies that the relative mass
flux $\rho u$ has a constant value denoted by $m=\rho_L u_L$. As
already underlined, we focus ourselves on traveling wave solutions
associated with the first GNL field; namely we consider $m>0$. Observe
that the Lax condition (\ref{lax}) expressed expressed for a null
velocity $\sigma$ reads 
$$
m>\rho_L c_L.
$$

Next by integrating once the second equation in (\ref{sysG1}), the
identity $u=m \tau$ allows one to derive the following
$(N+1)$-dimensional autonomous system:
\be\label{syssigma}
\aligned
\displaystyle
\dot{\tau}&=\frac{1}{\mu(\tau,\vs)} \Big(
   p(\tau,\vs)-p(\tau_L,\vs_L)+m^2(\tau-\tau_L)\Big)
   := \frac{1}{\mu(\tau,\vs)} \Fcal(\tau,\vs),\\
\displaystyle
\dot{s}_i&=\frac{\mu_i(\tau,s_i)}{\mu^2(\tau,\vs)T_i(\tau,s_i)}\Fcal^2(\tau,\vs),
  \qquad 1\le i \le N,
\endaligned
\ee
where dots denote differentiation with respect to the rescaled
variable $x/m$ that we shall refer with little abuse as a time in the
sequel.

This dynamical system is endowed with the following open subset of
$\RR^{N+1}$ which will serve as a natural phase space:
\be\label{phasespace}
\Omega=\left\{\omega:=(\tau,\vs)\in\RR^{N+1};\tau>0\right\}.
\ee
To shorten the notation, a given function $\Psi$ of $\tau$ and $\vs$
is simply denoted hereafter by $\Psi(\omega)$.

Recall that the total viscosity $\mu(\omega)$ is assumed to
stay strictly positive over $\Omega$. Then, the regularity assumptions
made on all the thermodynamic and viscosity mappings ensure that the
vector field in (\ref{syssigma}) is continuously differentiable

The unique nonextensible solution of
(\ref{syssigma}) with initial data $\omega_0$ in $\Omega$ is referred
as to the flow $\omega_0\cdot t$ for the times $t$ in the maximal
interval of existence $(t^-(\omega_0),t^+(\omega_0))$. The positive
(respectively negative) semi-orbit $\gamma^+(\omega_0)$
(resp. $\gamma^-(\omega_0)$) classically denotes the set of states
$\omega_0\cdot[0,t^+(\omega_0))=\{\omega_0\cdot t : 0\le
t<t^+(\omega_0)\}$
(resp. $\omega_0\cdot(t^-(\omega_0),0]=\{\omega_0\cdot t :
t^-(\omega_0)<t\le 0\}$), the orbit being then defined as
$\gamma(\omega_0)=\gamma^-(\omega_0)\cup\gamma^+(\omega_0)$. At last,
for each $\omega_0$ in $\Omega$, the positive limit set (the
so-called $\varpi$-limit set in what follows) of $\omega_0$ finds the
definition
$\varpi(\omega_0):=\cap_{t>0}\overline{\gamma^+(\omega_0\cdot t)}$,
such a set is thus empty as soon as $t^+(\omega_0)$ is finite.

Before we enter the central part of the analysis, let us underline
that the $(N+1)$ constitutive variables of (\ref{syssigma}) are
necessarily kept in their evolution in time in a $N$-dimensional
sub-manifold of $\Omega$, the latter being entirely prescribed by the
choice of the initial data $\omega_0\in\Omega$. This is the matter of the following
statement which essentially reflects the conservation property met by
the total energy (\ref{totalE}).

\begin{proposition}\label{propHconstant}
Let $\omega_0$ be a given state in $\Omega$. Then the flow
$\omega_0\cdot t$ satisfies for all time in its maximal interval of
existence:
\be\label{energyconservation}
\Hcal(\omega_0\cdot t)=\Hcal(\omega_0),
\ee
where the regular mapping $\Hcal:\Omega\to \RR$ is defined by 
$$
\Hcal(\omega)=e(\omega)-e(\omega_L)-\frac{m^2}{2}(\tau^2-\tau_L^2)+(m^2\tau_L+p(\omega_L))(\tau-\tau_L).
$$
\end{proposition}

\begin{proof}
All the flows under consideration are at least continuously
differentiable in their maximal interval of existence. The additional
conservation law (\ref{totalE}) for the total energy therefore applies
and its differential form reads 
\be\label{dfte}
\left\{ (E +\tau p(\omega))\rho u\right\}_x=\left\{\frac{\mu}{2}(u^2)_x\right\}_x.
\ee
In view of the algebraic invariant $\rho u=m$, (\ref{dfte}) once
integrated for a prescribed $\omega_0$ in $\Omega$ between time zero
and a given time $t$ in $(t^-(\omega_0),t^+(\omega_0))$ can be seen to
read 
$$
\aligned
\left\{ E +\tau p(\omega))\rho u\right\}(t)-\left\{ E +\tau p(\omega))\rho u\right\}(0)
& = 
	\left\{\tau(\mu \dot{\tau})\right\}(t)-\left\{\tau(\mu \dot{\tau})\right\}(0),
	\\
& =  \left\{\tau\Fcal(\omega)\right\}(t)-\left\{\tau\Fcal(\omega)\right\}(0).
\endaligned
$$ 
Since $E$ writes $m^2\tau^2/2+\eps(\omega)$, the definition of
$\Fcal$ given in (\ref{syssigma}) easily yields the required identity
(\ref{energyconservation}) after some rearrangements in the terms
while subtracting for convenience to both sides the constant
$\eps_L+m^2\tau_L^2/2+\tau_Lp_L$.
\end{proof}

The above statement clearly implies that all the possible heteroclinic
orbits of (\ref{syssigma}) which connect the critical point $\omega_L$
in the past are only made of states $\omega$ such that 
\be\label{H0}
\Hcal(\omega)=\Hcal(\omega_L)=0.
\ee

To end up with these preliminary remarks, we point out an elementary
but useful geometrical property of the flows associated with
(\ref{syssigma}) which will put restriction on possible right
connecting states.
\begin{lemma}\label{leminvariant}
Let $\omega_0$ be given in $\Omega$, then the subset of $\Omega$ defined by
\be\label{omes0}
\Omega(\omega_0)=\left\{\omega\in\Omega; \vs\ge\vs_0,\Hcal(\omega)=\Hcal(\omega_0) \right\}
\ee
is positively invariant.
\end{lemma}
The invariance of this region with respect to all positive semi-flows
immediately follows from the non-negativeness of the $N^{\rm th}$-last
components of the vector field entering the definition of
(\ref{syssigma}). As a consequence, possible heteroclinic orbits
connecting $\omega_L$ in the past must entirely lie in 
\be\label{omeL}
\Omega(\omega_L)=\left\{\omega\in\Omega; \vs\ge\vs_L,\Hcal(\omega)=0 \right\}.
\ee
The region (\ref{omeL}) will play a central role in the derivation of
positively invariant compact sets.

Here, we exhibit some important features of the linearization
$LX(\omega_c)$ of the vector field $X$ at equilibrium points
$\omega_c$, i.e. at states satisfying $\Fcal(\omega_c)=0$. We check in particular
that such states are always non-hyperbolic points for which the space
$\RR^{N+1}$ writes as a direct sum of the eigenspaces associated with
$LX(\omega_c)$ under the following non-degeneracy condition:
\be\label{nondegener}
\del_{\tau}\Fcal(\omega_c)=m^2+\del_{\tau}p(\tau_c,\vs_c)\not=0.
\ee
In that aim, let us state some basic facts concerning the
linearization $LX(\omega_c)$. The requirement $\Fcal(\omega_c)=0$ is
easily seen to enforce all the partial derivatives of the $N^{\rm
  th}$-last components of $X$ to be identically zero (since these
components are all proportional to $\Fcal^2$). Under the nondegeneracy
condition (\ref{nondegener}), there consequently exists only one non
trivial eigenvalue namely $\del_{\tau}\Fcal(\omega_c)/\mu(\omega_c)$
while $\lambda=0$ is a semisimple eigenvalue of $LX(\omega_c)$ of
multiplicity $N$. Furthermore, the corresponding eigenspaces
$T(\omega_c)$ and $T^c(\omega_c)$ are respectively the span of $e_1$
and $e_2,...,e_N$ where $\{e_i\}_{1\le i\le N+1}$ stands for the
canonical orthonormal basis of $\RR^{N+1}$.

Equipped with these results, the center manifold theorem
ensures the existence of two locally invariant manifolds
$\Wcal(\omega_c)$ and $\Wcal^c(\omega_c)$ (the so-called center manifold)
of class at least $C^1$ and $C^0$ respectively which go through
$\omega_c$ and are respectively tangent to $T(\omega_c)$ and
$T^c(\omega_c)$ at this point. The regularity properties 
above are indeed inherited from the continuous differentiability
of the vector field, according to this theorem.

Assuming $\del_{\tau}\Fcal(\omega_c)$ to be positive (respectively
negative), {\sl i.e.} assuming the corresponding sign for the unique
nontrivial eigenvalue of $LX(\omega_c)$; $\Wcal(\omega_c)$ is
classically referred as to the unstable (resp. stable) manifold with
superscript $u$ (resp. $s$).  Recall that the unstable
manifold of a point is the manifold composed of the totality of the
orbits which tend exponentially fast to the point in negative time;
the stable manifold being defined conversely.
Then by well-known topological considerations, two rest points namely
$\omega_L$ and $\omega_R$ are connected by a heteroclinic orbit
$\gamma$ precisely if
$\gamma\subset\Wcal^u(\omega_L)\cap\Wcal^s(\omega_R)$.

An obvious requirement for the existence of a heteroclinic orbit
connecting $\omega_L$ in the past is then 
$$
\del_{\tau}\Fcal(\omega_L)=m^2+\del_{\tau}p(\omega_L)<0,
$$
but the validity of such an inequality is precisely the matter of the
Lax condition (\ref{lax}). Conversely, a possible connecting
point $\omega_R$ in the future is necessarily subject to the condition
$\del_{\tau}\Fcal(\omega_R)>0$.

Now and since the unstable manifold $\Wcal^u(\omega_L)$ is one
dimensional, there exists locally exactly two solutions of
(\ref{syssigma}) which approach $\omega_L$ as $t\to-\infty$. Arguing
about the property of $\Wcal^u(\omega_L)$  to be tangent to $e_1$, the
associated almost horizontal orbits approach $\omega_L$ from the two
opposite directions $\tau\ge \tau_L$ and $\tau\le \tau_L$. With clear
notation, $\gamma_>(\omega_L)$ (respectively $\gamma_<(\omega_L)$)
will denote the first (resp. the second) orbit.

The following assertion discards the solution converging to $\omega_L$
for negative times from the region $\tau\ge\tau_L$. Note that such a
result precisely precluded expansion shocks to admit viscous
profiles.
\begin{proposition}\label{nonhetero}
There is no heteroclinic orbit of the dynamical system
(\ref{syssigma}) in the domain $\Ncal:=\{\omega\in\Omega :\tau\ge\tau_L,
\vs\ge\vs_L\}$.
\end{proposition}
Consequently, only the second solution can give rise to a heteroclinic
orbit. Since the vector field $X:\Omega\to\RR^{N+1}$ is
Lipschitz-continuous, the uniqueness part of the celebrated
Picard-Lindel\"of theorem readily gives the following.

\begin{corollary}\label{corounik}
There exists at most one heteroclinic orbit of the dynamical system
(\ref{syssigma}) which connects $\omega_L$ in the past.
\end{corollary}

The proof of Proposition~\ref{nonhetero} will follow from the following
statement.

\begin{lemma}\label{lemfp}
Any given state $\omega$ distinct from $\omega_L$ in the set
$\{\tau\ge\tau_L, \vs\ge\vs_L\}$, obeys $\Fcal(\omega)>0$.
\end{lemma}

\begin{proof}
Observe that the positiveness assumption (\ref{hypH1}) on the
Gr\"uneisen numbers implies that for all $\omega$ in the region under
interest ($\vs\ge\vs_L$):
\be\label{sf}
f(\tau):=\Fcal(\tau,\vs_L)\le \Fcal(\omega)
\ee
with equality iff $\vs=\vs_L$. In particular, for $\tau=\tau_L$ we have 
$\Fcal(\tau_L,\vs)>0$ as soon as $\vs>\vs_L$. Next, considering
$\tau>\tau_L$, the following identity
$$
f'(\tau)=m^2+\del_\tau p(\tau,\vs_L) 
$$
clearly yields, under the assumption (\ref{hypH2}) of positive fundamental
derivatives, that $\del_{\tau\tau}^2 p(\tau,\vs)>0$ for all
$\omega\in\Omega$ andn therefore, 
$$
f'(\tau)\ge f'(\tau_L)=m^2-(\rho_Lc_L)^2>0
$$
thanks to the Lax condition (\ref{lax}). It immediately
follows that $f(\tau)\ge f(\tau_L)=\Fcal(\omega_L)=0$ as soon as
$\tau\ge\tau_L$ with equality to zero iff $\tau=\tau_L$. The
inequality (\ref{sf}) then gives the required conclusion.
\end{proof}

\begin{proof}[Proof of Proposition~\ref{nonhetero}]
Let $\vn_{\omega}$ be the unit inward normal at the following
hypersurfaces $\{\tau=\tau_L,\vs\ge\vs_L\}$ and
$\{\tau\ge\tau_L,\vs=\vs_L\}$ for all states in these sets. Note that
these sets are the lower part of the boundary of $\Ncal$. 
In view of the definition of the
vector field, 
 Lemma \ref{lemfp} implies that, for such states, 
 $X(\omega)\cdot\vn_{\omega}\ge 0$. As a (well-known) 
consequence, $\Ncal$ stays invariant for all positive semi-flows. The
required conclusion follows again from Lemma \ref{lemfp} which says
that no critical point exists in $\Ncal$.
\end{proof}

We now stress another important consequence of the local
properties of the phase portrait at the rest point $\omega_L$. By
opposition to the states in the orbit $\gamma_>(\omega_L)$; the second
orbit $\gamma_<(\omega_L)$ emanating from the region $\tau\le\tau_L$
is by definition made of states $\omega$ that at least when close
enough to but distinct from $\omega_L$ give rise to a compression:
namely locally $\Fcal(\omega)<0$ in view of the governing equation for
$\tau$. In turn, this simple observation implies that the viscous
profile under study must remain uniformly
compressive. This claim is a consequence of the following statement.

\begin{lemma}\label{lemI}
The following set 
\be\label{defI}
\Ical=\left\{ \omega\in\Omega : \tau<\tau_L,~\vs\ge\vs_L,~\Fcal(\omega)\le 0 \right\}
\ee
is positively invariant under the action of the dynamical system
(\ref{syssigma}).
\end{lemma}

\begin{proof}
The above assertion is trivial for states $\omega_0\in\Ical$ which
satisfy $\Fcal(\omega_0)=0$. Considering states $\omega_0$ with
the property $\Fcal(\omega_0)<0$, we observe that the positive semi-flow
$\omega_0\cdot t$ necessarily satisfies $\Fcal(\omega_0\cdot t)<0$ for
all time in $[0,t^+(\omega_0))$. Indeed assuming the existence of a
finite time $t_c$ in this interval with the property
$\Fcal(\omega_0\cdot t_c)=0$ would result in a critical point
$\omega_0\cdot t_c$ for the dynamical system (\ref{syssigma}). But by
the Lipschitz-continuity property of the vector field in $\Omega$, it
is well-known that such a point cannot be reached in finite time.
\end{proof}

The orbit $\gamma_<(\omega_L)$ is therefore trapped in the region
$\Ical$. We now establish that, in addition, this orbit must remain
 within a compact subset $K$ of $\Ical$. This will imply that
$\gamma_<(\omega_L)$ is relatively compact. Well-known considerations
imply that the associated
$\varpi$-limit set is nonempty, compact and connected. The existence
of $K$ primary stems from the following result.

\begin{lemma}\label{lemvide}
Let $\omega_0$ be a given state in $\Omega$. Then the positive
semi-orbit $\gamma^+(\omega_0)$ has no limit point in the set
$\{\tau=0\}$.
\end{lemma}
This assertion immediately gives that the orbit $\gamma_<(\omega_L)$
has the same property.
\begin{proof}
For all time $t$ in $[0,t^+(\omega_0))$, the positive semi-flow
$\omega_0\cdot t$ is known to obey $\Hcal(\omega)=\Hcal(\omega_0)<\infty$
and $\vs\ge\vs_0$. Arguing about the positivity of all the temperatures
$T_i=\del_{s_i}\Hcal$, we immediately get 
$$
h(\tau):=\Hcal(\tau,\vs_0)\le \Hcal(\tau,\vs),
$$
with the property that $h(\tau)$ goes to infinity as $\tau$ goes to
zero (see the asymptotic condition (\ref{hypH3})).

Assume that $\overline{\gamma^+(\omega_0)}\cap\{\tau=0\}$ is
nonempty. As a consequence, for all $\varepsilon>0$ there exists
$t_{\varepsilon}\in (0,t^+(\omega_0))$ such that
$0<\tau|_{\omega_0\cdot t_{\varepsilon}}<\varepsilon$. Necessarily
there exists $\varepsilon_0>0$ so that $h(\tau|_{\omega_0\cdot
  t_{\varepsilon_0}})>\Hcal(\omega_0)$ and this rises the contradiction
with the preservation of $\Hcal(\omega_0)$ along the orbit.
\end{proof}

We have proven that any given positive semi-flow
of the dynamical system (\ref{syssigma}) with initial data $\omega_0$
in $\Omega$ satisfies $t^+(\omega_0)=\infty$ (since $\vs>\vs_0$).
We now conclude with the existence (and therefore uniqueness) of the
required viscous profile.

\begin{proposition}\label{propfinex}
There exists a state $\omega_R$ in $\Hcal^{-1}(0)\cap\Fcal^{-1}(0)$ which
is connected by $\gamma_<(\omega_L)$ in the future.
\end{proposition}
\begin{proof}

We first establish that the specific entropy vector $\vs$ stays
upper-bounded along all positive semi-flows with initial data in
$\Ical$. For fixed $\tau$ in $(0,\tau_L)$, the conditions
(\ref{Tpositif}) and (\ref{hypH1}) shows that $\Hcal(\tau,\vs)$ rises
arbitrarily with $\vs$. The
same steps as the ones involved in the previous proof, apply to give
the required result. As a consequence, $\gamma_<(\omega_L)$ must be
 included in a compact subset, namely $K$, of the
positively invariant region $\Ical$. This orbit is, therefore, relatively
compact and its $\varpi$-limit set is non-empty. This limit set must be
 included in $\Hcal^{-1}(0)$. To conclude,  observe
that $\tau$, when understood as a mapping of the variable $\omega$, 
trivially yields a Lyapunov function on $\Ical$ where, by
construction, $\Fcal(\omega)\le 0$. The LaSalle invariance principle
applied in connection with this Lyapunov function then ensures that
the non-empty $\varpi$-limit set is included in $\{\omega\in\Ical :
\Fcal(\omega)=0\}$. This establishes the existence of $\omega_R$.
\end{proof}


\section{End states for viscous layers with varying
  viscosity}
\label{EXP-0}

The existence (and uniqueness up to translation) of traveling wave
solutions to the multi-pressure Navier-Stokes equations was established
 in the previous section for $N$ viscosity laws
satisfying the non degeneracy condition (\ref{condvisco}). Being
given a fixed state $\omega_L\in\Omega$ and a velocity $\sigma$
according to the condition (\ref{lax}), we aim here at characterizing
the subset of $\Omega$ made of all the states $\omega_R$ that can be
reached in the future by a traveling wave with speed $\sigma$ and
connecting $\omega_L$ in the past. We naturally expect the exit state
$\omega_R$ to depend on the specific form of the $N$-uple
of viscosity laws. The dynamical system (\ref{syssigma}) shows that
such a dependence is in the ratios of the
viscosity laws. As a consequence, possible states $\omega_R$ to be
reached in the future from a fixed $\omega_L$ (at some speed
$\sigma$) generically depend on $N-1$ degrees of freedom. The set of
exit states we are seeking is thus expected to have $(N-1)$-dimension.

It will be convenient to study the projection of this set onto the
following positive cone of $\RR^N$ (understod as the space of the
specific entropies $\vs=(s_1,...,s_N)$) with origin $\vs_L$:
\be\label{def_cone}
S^+(\vs)=\left\{
\vs\in\RR^N ~/~ \vs=\vs_L+\lambda \va,~\va\in
\Scal_+^N,~\lambda\ge 0
\right\}.
\ee
Here, $\Scal^N_+$ denotes the (positive) part of the unit sphere in
$\RR^N$ defined by
$$
\Scal^N_+ = \left\{ \va\in\RR^N_+;~\parallel \va  \parallel = 1\right\}.
$$

For all possible entropies $\vs_R$ in the cone (\ref{def_cone}), the existence
 simply comes from the property that the
heterocline solutions of Theorem~\ref{theo_main} obey
$\vs_R\ge \vs_L$ and a strict inequality holds for (at least) one specific
entropy $s_i$ ($1\le i\le N$).

We show hereafter that the projection in the half cone
(\ref{def_cone}) of the states
$\omega_R$ that can be reached when varying the definition of the $N$
viscosity laws, is a smooth manifold with co-dimension one:
\be\label{newdefCcal}
\Ccal=\big\{\vs\in S^+(\vs) ~/~\vs= \vs_L + \Lambda_0(\va)\va,~\va\in \Scal^N_+\big\},
\ee
for some suitable mapping $\Lambda_0(\va):\va\in \Scal^N_+ \mapsto \Lambda_0(\va)\in\RR$
which precise definition will be given latter on. The derivation of
the proposed manifold is performed in two steps. In a first step, we
analyse closely all the critical points $(\tau_c,\vs_c)$ of the
dynamical system (\ref{syssigma}), {\it i.e.} the solutions of 
\be
\label{eq1critique}
\Fcal(\tau_c,\vs_c) =0,
\ee
without reference to a precise $N$-uple of viscosity laws. 

We
emphasize that eligible critical points that can be reached from the
state $\omega_L$ in the past must preserve the total energy as stated
in (\ref{H0}). Such states must therefore solve in addition 
\be
\label{eq2critique}
\Hcal(\tau_c,\vs_c) =0,\quad\mbox{with}\quad
\vs_R\ge \vs_L.
\ee
Analyzing the solution of (\ref{eq1critique})-(\ref{eq2critique}) will
give birth to the manifold (\ref{newdefCcal}).

In a second step, we will establish that any given value $\vs$ in the
proposed manifold can be actually achieved for at least one suitable
$N$-uple of viscosity laws. As a consequence, the manifold
(\ref{newdefCcal}) is entirely made of all the specific entropy
$\vs_R$ that can be reached in the future by a traveling wave
solution with speed $\sigma$ and issued from $\omega_L$, when varying
the definition of the $N$ viscosity laws. 

Let us outline the content
of this section. We first analyze the mappings $\vs\in
S^+(\vs)\mapsto\tau_{\Fcal}(\vs)\in\RR_+$ that solve 
\be
\label{eq3critique}
\Fcal(\tau_{\Fcal}(\vs),\vs) =0.
\ee
We then characterize the mapping $\vs\in
S^+(\vs)\mapsto\tau_{\Hcal}(\vs)\in\RR_+$ solving
\be
\label{eq4critique}
\Hcal(\tau_{\Hcal}(\vs),\vs) =0.
\ee
Equipped with these two families of functions we will study for
existence values of the specific entropy $\vs_c$ in $S^+(\vs)$ that
satisfy $\tau_{\Fcal}(\vs_c)=\tau_{\Hcal}(\vs_c)$, namely values of
$\vs$ that simultaneously solve (\ref{eq3critique}) and
(\ref{eq4critique}). We now state the main result of this section.
\begin{theorem}
\label{main_cp}

Assume that (\ref{thermo1})-(\ref{hypH5}) on the
thermodynamics are satisfied. Then there exists a
unique map $\Tcal \in \Ccal^0\big({\bar \Kcal},\RR^*_+\big)\cap 
{\Ccal}^1\big(\Kcal,\RR^*_+\big)$ where $\Kcal\subset S^+(\vs)$ reads 
\be
\Kcal = \big\{ \vs\in S^+(\vs) ~/~ \vs = \vs_L + \lambda \va, ~ \va
\in \Scal^N_+,~\lambda\in ]0, \Lambda_0(\va)[\big\},
\label{defK}
\ee
for some smooth application $\Lambda_0 \in
\Ccal^1\big(\Scal^N_+,\RR^*_+\big)$ with the following properties:
\begin{enumerate}
\item[(i)]  $ \Hcal \big(\Tcal(\vs),\vs\big)=0$, for all $\vs \in {\bar \Kcal}$,
\item[(ii)] $ \Fcal \big(\Tcal(\vs),\vs\big)=0$, for all $\vs\in \Ccal$
  where
\be
\label{defCcal}
\Ccal = \big\{ \vs\in S^+(\vs) ~ /~ \vs=\vs_L + \Lambda_0(\va)\va,~\va\in \Scal^N_+\big\}.
\ee
\end{enumerate}
In addition $\Tcal$ obeys
\begin{enumerate}
\item[(iii)] $\Fcal \big(\Tcal(\vs_L),\vs_L\big)=0$,
\item[(iv)]  $\Fcal \big(\Tcal(\vs),\vs\big) < 0$, for all $\vs \in \Kcal$.
\end{enumerate}
\end{theorem}

The mapping $\Lambda_0:\Scal^N_+\to\RR^*_+$ will be built in the
course of the proof. Rephrasing the above result, the function
$\Tcal(\vs)$ with $\vs\in\Ccal$, simultaneously makes vanish $\Fcal$
and $\Hcal$, so that all the values $\vs$ in the smooth manifold
$\Ccal$ are candidate for being reached in the future from the state
$\omega_L$ via a traveling wave with speed $\sigma$ for suitable
choice of the $N$ viscosity laws. 

We now show that all the values $\vs$ in the manifold
$\Ccal$, defined by (\ref{defCcal}), are actually eligible candidates
for entering the definition of the specific entropy in exit states
$\omega_R$.

\begin{lemma}
A state $\omega_L$ being given in $\Omega$ and a velocity $\sigma$
being prescribed according to (\ref{lax}). For any given
$\vs\in\Ccal$, there exists at least one relevant definition of the
$N$-uple of viscosity laws which yields a traveling wave solution with
speed $\sigma$ issued from $\omega_L$ and connecting a state
$\omega_R$ in the future with $\vs_R=\vs$.
\end{lemma}

\begin{proof}
The proof of this result makes use of particular viscosity laws under
the form 
\be
\label{eq7visco}
\mu_i(\tau,s_i)=\mu_i^0 T_i(\tau,\vs_i),
\quad
\mu_i^0\ge 0,
\quad 1\le i\le N.
\ee
The non degeneracy condition (\ref{condvisco}) is satisfied as soon as 
\be
\label{condviscostar}
\sum_{i=1}^N \mu_i^0 >0,
\ee
since each of the temperature law $T_i(\tau,\vs_i)$ is assumed to be
positive. Without lost of generality, we assume $\mu_N^0>0$.

We stress that viscosity laws which linearly depend on the
temperature naturally arise in the kinetic theory for dilute gases. We
refer the reader to the book by Hirschefelder, Curtiss
and Bird \cite{hirschefelder}.

Observe that viscosity laws in the special form (\ref{eq7visco}) let
evolve each specific entropy according to
$$
\dot{s}_i=\frac{\mu_i^0}{\mu^2(\tau,\vs)}\Fcal^2(\tau,\vs),
  \qquad 1\le i \le N.
$$
We thus infer the following  $(N-1)$ balance equations linking
the evolution of the first $(N-1)$ specific entropies $s_i$ to the
last one:
$$
\dot{s}_i=\frac{\mu_i^0}{\mu_N^0}\dot{s}_N,
  \qquad 1\le i \le N-1.
$$
Since the ratios $\mu_i^0/\mu_N^0$ are constant real numbers,
we deduce:
$$
\left(s_i-\frac{\mu_i^0}{\mu_N^0}s_N\right)(\xi)
=s_i^L-\frac{\mu_i^0}{\mu_N^0}s_N^L,
\qquad\mbox{for all}\quad
\xi\in\RR.
$$
We therefore end up with $(N-1)$ jump relations 
\be
\label{eq8si}
s_i^R-s_i^L = \frac{\mu_i^0}{\mu_N^0} \left( s_N^R-s_N^L \right),
\quad 1\le i \le N.
\ee

We emphasize at this stage that $s_N^R-s_N^L>0$ in view of our
assumption $\mu_N^0>0$. From the jump relation (\ref{eq8si}), we
therefore get 
\be
\label{eq9vs}
\vs_R-\vs_L = \frac{s_N^R -s_N^L}{\mu_N^0}
\left(\begin{array}{c}
\mu_1^0 \\ \vdots \\ \mu_N^0
\end{array}\right),
\ee
which is obviously in the form $\Lambda_0(\va)\va$, for $\va$ in
$\Scal^N_+$, given by
$$
\va = \frac{1}
{\sqrt{\sum_{j=1}^N (\mu_j^0)^2}} \left( \mu_j^0 \right)_{1\le j \le N}.
$$

Next and up to some relabelling in the viscosity in order to
allow $\mu_N^0$ to vanish, any given $\va\in\Scal^N_+$ gives rise to
an admissible $N$-uple of viscosity coefficients.
This concludes the proof.
\end{proof}

\begin{remark}
The identity (\ref{eq9vs}) shows in addition that the mapping
$\va\mapsto\Lambda_0(\va)$ can be built as soon as the jump in the
last specific entropy $s_N^R-s_L^N$ is known.
This evaluation can be performed numerically; see, for instance,
\cite{BerthonCoquel1}.
\end{remark}

We now give a proof of the main result of this section, namely Theorem~\ref{main_cp}. In that aim, and as already claimed, we propose to
first study for existence in $\Scal^N_+$ the roots $\tau(\vs)$ of
$\Fcal(\tau,\vs)=0$. Then, we shall study their distinctive properties
by investigating the values of $\Hcal(\tau(\vs),\vs)$. 

\begin{proposition}
\label{prop1_cp}
There exists two maps, we denote by $\tau^\pm$ belonging to
$\Ccal^1\big({\Dcal}\cup\{\vs_L\},\RR^*_+)\cap\Ccal^0\big({\bar
  D},\RR^*_+)$, where $\Dcal$ is the subset of $S^+(\vs_L)$ defined by 
\be
\label{def_D}
\Dcal = \big\{ \vs\in S^+(\vs_L) ~/~ \vs = \vs_L + \lambda \va,~ \va\in \Scal^N_+,~ \lambda \in ]0, {\bar \Lambda}(\va)[\big\},
\ee
for some ${\bar \Lambda}\in\Ccal^1\big(\Scal^N_+,\RR^*_+\big)$ with 
${\bar \Lambda}(\va) > \Lambda_0(\va)$ for all $\va\in \Scal^N_+$, so that 
$$
\Fcal(\tau^\pm(\vs),\vs) = 0, \quad \quad \vs\in{\bar \Dcal}.
$$
In addition, these two families of roots are interlaced according to 
\begin{enumerate}
\item[(i)]   $\tau^-(\vs) < \tau^+(\vs) < \tau_L$, for all $\vs \in
  \Dcal\backslash \{\vs_L\}$,
\item[(ii)]  $\tau^+(\vs) = \tau_L$ in ${\bar \Dcal}$ iff $\vs =
  \vs_L$ with $\tau^-(\vs_L)<\tau^+(\vs_L)=\tau_L$,
\item[(iii)] $\tau^-(\vs)=\tau^+(\vs)$ in ${\bar \Dcal}$ iff $\vs =
  \vs_L + {\bar \Lambda}(\va) \va$, $\va\in \Scal^N_+$.
\end{enumerate}

\end{proposition}
Again, the map ${\bar \Lambda} : \Scal^N_+\to \RR^*_+$ will be built
in the course of the proof. But from now on, notice that ${\bar
  \Kcal}\subset\Dcal$. We shall show that for fixed $\vs\in{\bar \Dcal}$,
$\Fcal(\tau,\vs)=0$ only admits $\tau^\pm(\vs)$ as roots and cannot be
solved in $\tau$ for values of $\vs$ in $S^+(\vs_L)\backslash
{\bar \Dcal}$. As a consequence, all the critical points
$(\tau(\vs),\vs)$ of (\ref{syssigma}) are necessarily achieved for
$\vs\in{\bar \Dcal}$ so that $\tau(\vs)$ must coincide with either
$\tau^-(\vs)$ or $\tau^+(\vs)$ for suitable values of $\vs\in{\bar \Dcal}$
: {\it i.e.} such that $\Hcal(\tau^-(\vs),\vs)=0$ or
$\Hcal(\tau^+(\vs),\vs)=0$. In this way, let us state some properties
of $\Hcal$ with respect to the above two families of roots.

\begin{proposition}
\label{prop2_cp}
Using the notation in Propositions \ref{main_cp} and \ref{prop1_cp},
we have 
\begin{enumerate}
\item[(i)]  $\Hcal(\tau^+(\vs_L),\vs_L) = 0$,
\item[(ii)] $\Hcal(\tau^+(\vs),\vs) > 0$, for all $\vs\in {\bar
  \Dcal}\backslash\{\vs_L\}$,
\end{enumerate}
while
\begin{enumerate}
\item[(iii)] $\Hcal(\tau^-(\vs),\vs) < 0$, for all $\vs\in \Kcal\cup\{\vs_L\}$,
\item[(iv)]  $\Hcal(\tau^-(\vs),\vs) = 0$, for all $\vs \in
  \Ccal=\big\{\vs\in\Dcal~/~\vs= \vs_L + \Lambda_0(\va)\va,~\va\in
  \Scal^N_+\big\}$,
\item[(v)]   $ \Hcal(\tau^-(\vs),\vs) > 0$, for all $\vs\in {\bar \Dcal}\backslash {\bar \Kcal}$.
\end{enumerate}
\end{proposition}

Put in other words, the critical points of the differential system
(\ref{syssigma}) necessarily coincide with the set
$\big\{\tau^+(\vs_L),\vs_L\big\}$ and
$\big\{(\big(\tau^-(\vs),\vs\big)~/~\vs\in \Ccal\big\}$. Keeping this
in mind, we next analyze the roots $\tau(\vs)$ of
$\Hcal(\tau(\vs),\vs)$. The following claim states that $\Hcal$ admits
three distinct branches of roots in ${\bar \Kcal}$. A particular
attention is paid to single out a branch $\Tcal$ obeying the
requirements:
\be
\label{TsL}
\Tcal(\vs_L) = \tau^+(\vs_L)
\qquad\mbox{together with}\qquad
\Tcal(\vs)=\tau^-(\vs)
\qquad\mbox{for all}\quad
\vs\in\Ccal,
\ee
as put forward in Proposition \ref{prop2_cp}.

\begin{proposition}
\label{prop3_cp}
There exist three maps in $\Ccal^0({\bar K},\RR^*_+)\cap\Ccal^1\big(\Kcal,\RR^*_+\big)$ respectively denoted by ${\check \Tcal},~\Tcal,
~{\hat \Tcal} : {\bar \Kcal}\to \RR^*_+$, so that:
\begin{enumerate}
\item[(i)] $\Hcal(\Tcal(\vs),\vs) = \Hcal({\check \Tcal}(\vs),\vs) = \Hcal({\hat \Tcal}(\vs),\vs) = 0$,
 for all $\vs\in {\bar \Kcal}$.
\end{enumerate}
These are interlaced with the roots $\tau^\pm(\vs)$ of $\Fcal$ as
follows:
\begin{enumerate}
\item[(ii)]  ${\check \Tcal}(\vs) < \tau^-(\vs) < \Tcal(\vs) <
  \tau^+(\vs) < {\hat \Tcal}(\vs)$, for all $\vs\in\Kcal$,
\item[(iii)] ${\check \Tcal}(\vs_L) < \Tcal(\vs_L) = \tau^+(\vs_L) = {\hat \Tcal}(\vs_L)$, 
\item[(iv)]  ${\check \Tcal}(\vs) = \tau^-(\vs) = \Tcal(\vs) < {\hat \Tcal}(\vs)$, for all $\vs\in\Ccal$.
\end{enumerate}
\end{proposition}

Observe that the intermediate mapping $\Tcal$ fulfills the
requirements (\ref{TsL}) so that Theorem~\ref{main_cp} is
established.

We now give the proofs of Propositions \ref{prop1_cp} to
\ref{prop3_cp}. Proposition \ref{prop1_cp} relies on the following two
technical lemma.

\begin{lemma}
\label{lem1_cp}
For all fixed $\vs\in S^+(\vs_L)$, $\Fcal(.,\vs)$ admits a unique
minimum in $\tau$ we denote ${\bar \tau}(\vs)$ where ${\bar
  \tau}\in\Ccal^1\big( S^+(\vs_L),\RR^*_+\big)$ with ${\bar \tau}(\vs)
< \tau_L$ for all $\vs\in {\bar \Dcal}$. This minimum obeys:
\begin{enumerate}
\item[(i)]  $\Fcal({\bar \tau}(\vs),\vs) < 0$, for all $\vs\in \Dcal\cup\{\vs_L\}$,
\item[(ii)] $\Fcal({\bar \tau}(\vs),\vs) = 0$, for all $\vs\in \Gamma :=
  \{\vs\in S^+(\vs_L)~/~\vs=\vs_L+{\bar\Lambda}(\va)\va,~\va\in
  \Scal^N_+\}$, 
\item[(iii)] $\Fcal({\bar \tau}(\vs),\vs) > 0$, for all $\vs\in S^+(\vs_L)\backslash{\bar \Dcal}$,
\end{enumerate}
where the set $\Dcal$ has been defined in Proposition \ref{prop1_cp}.
\end{lemma}

\begin{lemma}
\label{lem2_cp}
For all fixed $\vs\in S^+(\vs_L)$, $\Fcal(.,\vs)$ is strictly decreasing (respectively strictly increasing)
for all $\tau\in (0, {\bar\tau}(\vs))$ (resp. for all $\tau > {\bar\tau}(\vs)$) and achieves the following limits
$$
\lim_{\tau \to 0^+}\Fcal(\tau,\vs)=+\infty,\quad \lim_{\tau\to +\infty}\Fcal(\tau,\vs)=+\infty.
$$
As a consequence, $\Fcal(\tau,\vs)=0$ can be solved in $\tau$ only
when $\vs\in{\bar \Dcal}$, with exactly one solution when $\vs\in
\Gamma$ and exactly two solutions $\tau^\pm(\vs)$ for $\vs\in {\bar
  \Dcal}\backslash \Gamma$. These solutions define two maps $\tau^\pm
\in \Ccal^0({\bar
  \Dcal},\RR_+^*)\cap\Ccal^1\big(\Dcal\cup\{\vs_L\},\RR_+^*\big)$ with
the following properties:
\begin{enumerate}
\item $\tau^-(\vs) < {\bar \tau}(\vs) < \tau^+(\vs) < \tau_L$, for all $\vs\in \Dcal$,
\item $\tau^-(\vs) = {\bar \tau}(\vs) = \tau^+(\vs) < \tau_L$, for all $\vs = \vs_L + {\bar \Lambda}(\va)\va,~\va\in \Scal^N_+$,
\item $\tau^-(\vs_L) < {\bar \tau}(\vs_L) < \tau^+(\vs_L) = \tau_L$.
\end{enumerate}
\end{lemma}

We now establish Lemma \ref{lem1_cp} underlining that the set $\Dcal$
entering the Proposition \ref{prop2_cp} will be explicitly derived
in the course of the proof.

\begin{proof}[Proof of Lemma \ref{lem1_cp}]
Let $\vs$ be fixed in $S^+(\vs)$. Arguing about the smoothness of the
internal energies, the map $\tau \mapsto \Fcal(.,\vs)$ is at least of class
$\Ccal^2(\RR^*_+)$. Easy calculations then yield for all $\tau > 0$:
$$
\frac{\partial \Fcal}{\partial \tau}(\tau,\vs)=\frac{\partial p}{\partial \tau}(\tau,\vs)+m^2,\nonumber\\
\frac{\partial^2 \Fcal}{\partial \tau^2}(\tau,\vs)=\frac{\partial^2 p}{\partial \tau^2}(\tau,\vs).\nonumber
$$
On the one hand, the map $\tau\mapsto\frac{\partial \Fcal}{\partial
 \tau}(\tau,\vs)$ is strictly increasing in view of the genuine
nonlinearity assumption (\ref{hypH2}) for the total pressure. On the other hand, 
assumptions (\ref{hypH4}) on the asymptotic behaviour of
$\frac{\partial p}{\partial \tau}$  imply that:
$$
\lim_{\tau\to 0^+}\frac{\partial \Fcal}{\partial \tau}(\tau,\vs)=-\infty~\mbox{ and }~
\lim_{\tau\to +\infty}\frac{\partial \Fcal}{\partial \tau}(\tau,\vs)=m^2>0.
$$

As a consequence, for all $\vs\in S^+(\vs_L)$, there exists a unique $\bar{\tau}(\vs)>0$ so that 
$\frac{\partial \Fcal}{\partial \tau}(\bar{\tau}(\vs),\vs)=0$. This defines a map ${\bar \tau}$ in 
$\Ccal^1\big( S^+(\vs_L),\RR^*_+\big)$ thanks to the implicit function
theorem. Note that the assumption (\ref{lax}) on the relative Mach
number implies that $\partial_\tau\Fcal(\tau_L,\vs_L) > 0$ while
$\partial_\tau\Fcal({\bar \tau}(\vs_L),\vs_L)=0$, therefore hypothesis
(\ref{hypH1}) ensures:
\be
\label{lttl}
{\bar \tau}(\vs_L)<\tau_L.
\ee
Next, we construct the set $\Dcal\subset S^+(\vs_L)$ introduced in
Proposition \ref{prop2_cp} when studying for existence the zeros of
$\vs\in S^+(\vs_L) \mapsto \Fcal({\bar \tau}(\vs),\vs)$. In this way, we
first notice that by definition of ${\bar \tau}(\vs)$, for all
$\vs\in S^+(\vs_L)$ we have:
$$
\Fcal({\bar \tau}(\vs),\vs) = p({\bar \tau}(\vs),s) - p(\tau_L,\vs_L) - {{\partial p}\over{\partial \tau}}({\bar \tau}(\vs),\vs)\big({\bar \tau}(\vs)-\tau_L).
$$
Introducing the auxiliary function $\phi : \RR_+\times\Scal^N_+ \to S^+(\vs_L)$ defined by 
$$
\phi(\lambda,\va) = \Fcal\big({\bar \tau}(\vs_L + \lambda \va), \vs_L+\lambda \va\big)
$$
straightforward calculations give 
\be
\label{derphi}
{{\partial \phi}\over{\partial \lambda}}(\lambda,\va) = \sum_{1\le i\le N}\left(
{{\partial p_i}\over{\partial s_i}}a_i + 
\left(\sum_{1\le j\le N} {{\partial^2 p_j}\over{\partial \tau^2}}
{{\partial {\bar \tau}}\over{\partial s_i}}+
{{\partial^2 p_i}\over{\partial \tau \partial s_i}}\right)a_i\right).
\ee

But differentiating the identity $\partial_\tau p({\bar
  \tau}(\vs),\vs) = - m^2$, valid for all $\vs \in S^+(\vs_L)$, easily implies that (\ref{derphi}) reduces to 
\be
{{\partial \phi}\over{\partial \lambda}}(\lambda,\va) = \sum_{1\le i\le N}
{{\partial p_i}\over{\partial s_i}}\left({\bar \tau}(\vs_L+\lambda\va),\vs_L+\lambda\va\right)a_i.
\ee
This derivative is therefore strictly positive for all $\va\in
\Scal^N_+$, as follows from (\ref{hypH1}). Next, arguing about the strict
convexity in $\tau$ of the total pressure and the property
(\ref{lttl}) expressing that ${\bar \tau}(\vs_L) < \tau_L$, we get 
$$
\phi(0,\va) = p({\bar \tau}(\vs_L),\vs_L)-p(\tau_L,\vs_L)-{{\partial p}\over{\partial \tau}}({\bar \tau}(\vs_L),\vs_L)\big({\bar \tau}(\vs_L)-\tau_L\big) < 0.
$$

To conclude, we need to check that for all $\va\in
\Scal^N_+$, the map $\lambda \mapsto \phi(\lambda,\va)$ achieves positive
values for finite values of $\lambda$. Indeed, the implicit function
theorem will thus ensure the existence of a map ${\bar \Lambda} \in
\Ccal^1\big(\Scal^N_+,\RR^*_+\big)$ well defined in $\Dcal$ 
with the following property:
\be
\label{barL}
\phi({\bar \Lambda}(\va),\va) = 0,
\qquad\mbox{for all}\quad \va\in \Scal^N_+
\ee
with $\phi(\lambda, \va) < 0$ (respectively $>0$) for all $\lambda <
{\bar \Lambda}(\va)$ (resp. $\lambda >{\bar \Lambda}(\va)$). This is
nothing but the required result. To establish the validity of
(\ref{barL}), we show that for all $\va\in \Scal^N_+$ there exists
$\lambda_\star(\va)>0$ so that 
\be
\label{egtl}
{\bar \tau}(\vs_L + \lambda_\star(\va)) = \tau_L.
\ee
Indeed for such values of $\lambda$, $\phi$ boils down to 
$$
\phi(\lambda_\star(\va),\va) = p(\tau_L,\vs)-p(\tau_L,\vs_L) > 0,
$$
as follows from (\ref{hypH1}). To derive (\ref{egtl}),
we introduce the auxiliary smooth function $\theta \in
\Ccal^1\big(\RR_+,\RR\big)$ defined for all $\va\in \Scal^N_+$ by 
$$
\theta(\lambda) = {{\partial p}\over{\partial \tau}}(\tau_L,\vs_L + \lambda\va) + m^2.
$$

For any given $\va$ in $\Scal^N_+$, we establish the existence of
solutions to $\theta(\lambda)=0$, $\lambda_\star(a)$ being chosen to
be for instance the smallest one for definiteness. Existence of such
solution(s) readily follows from 
the assumption (\ref{lax}) on the relative Mach number ensuring that
$\theta(0) > 0$ while the asymptotic condition (\ref{hypH4}) ensures
$\theta(\lambda)<0$ for large enough $\lambda$. Note that the
solutions under consideration are strictly positive. In addition and
since ${\bar \Lambda}(\va) < \lambda_\star(\va)$ for all $\va\in
\Scal^N_+$, we obtain 
\be
\label{btlttl}
{\bar \tau}(\vs) < \tau_L, \quad \quad \vs\in {\bar \Dcal},
\ee 
which concludes the proof of Lemma \ref{lem1_cp}.
\end{proof}

\begin{proof}[Proof of Lemma \ref{lem2_cp}]
Let be given $\vs$ in $ S^+(\vs_L)$. By definition of ${\bar
  \tau}(\vs)$, $\Fcal(.,\vs)$ achieves the monotonicity properties
stated in Lemma \ref{lem2_cp}, the required limits immediately
follows from the asymptotic conditions (\ref{hypH3}). 
The study of the sign of $\Fcal({\bar \tau}(\vs),\vs)$ 
(as we 
  proposed earlier) obviously implies that for fixed $\vs\in
S^+(\vs_L)$, the equation $\Fcal(\tau,\vs)=0$ has exactly two
solutions $\tau^-,~\tau^+$ in ${\bar \Dcal}\backslash \Gamma$ so that
$\tau^- < {\bar \tau}(\vs) < \tau^+$; this equation has exactly one solution, namely
${\bar \tau}(\vs)$, when $\vs \in \Gamma$ and has no solution whenever
$\vs\in S^+(\vs_L)\backslash{\bar \Dcal}$. In addition, using the
notation introduced in the proof of Lemma \ref{lem1_cp}, it can be
easily seen that the following limits hold true:
\be
\label{coincide}
\lim_{\lambda \to {\bar \Lambda}(\va)^-} \tau^\pm(\vs_L+\lambda \va) =
    {\bar \tau}(\vs_L+{\bar \Lambda}(\va)\va),
\qquad\mbox{for all}\quad \va\in\Scal^N_+
\ee
These observations allow for the definition of two maps $\tau^\pm : {\bar \Dcal}\to \RR^*_+$ satisfying:
$$
\Fcal(\tau^\pm(\vs),\vs)=0, \quad \hbox{for all } \quad \vs\in {\bar \Dcal},
$$
and so that 
\be
\label{lip}
\tau^-(\vs) < {\bar \tau}(\vs) < \tau^+(\vs), \quad \vs \in {\bar \Dcal}\backslash \Gamma; \quad \tau^-(\vs) = {\bar \tau}(\vs) = \tau^+(\vs), \quad \vs\in \Gamma.
\ee

Then the above inequalities yield $\partial_\tau
\Fcal(\tau^\pm(\vs),\vs) \not= 0$ for all $\vs\in \Dcal\cup\{\vs_L\}$
in view of (\ref{lip}) so that the implicit function theorem ensures
that $\tau^\pm\in \Ccal^1(\Dcal\cup\{\vs_L\},\RR^*_+)$ while
(\ref{coincide}) gives that $\tau^\pm\in \Ccal^0({\bar \Dcal},\RR^*_+)$ .

Next, focusing to some given $\vs\in {\bar \Dcal}\backslash \{\vs_L\}$, we observe that
$$
\Fcal(\tau_L,\vs) = p(\tau_L,\vs)-p(\tau_L,\vs_L) > 0
$$ 
so that necessarily, either $\tau_L < \tau^-(\vs)$ or $\tau^+(\vs) <
\tau_L$. In addition, the identity $\Fcal(\tau_L,\vs_L)=0$ expresses
that either $\tau^-(\vs_L)=\tau_L$ or $\tau^+(\vs_L)=\tau_L$. But
Lemma \ref{lem1_cp} ensures that ${\bar \tau}(\vs) < \tau_L$ for all
$\vs\in {\bar \Dcal}$. This concludes the proof.
\end{proof}

Equipped with these two lemmas, the proof of Proposition \ref{prop1_cp} is 
essentially completed: the required inequality 
${\bar \Lambda}(\va) > \Lambda_0(\va)$ for all $\va\in \Scal^N_+$ will be deduced from the derivation of the set $\Kcal$
we propose hereafter. We will need the following technical result.

\begin{lemma}
\label{lem3_cp}
For all $\vs\in\Gamma$, $\Hcal({\bar \tau}(\vs),\vs) > 0$.
\end{lemma}

This statement actually indicates that there is no critical point on $\Gamma$.

\begin{proof}
To shorten the notation, let us introduce 
$$
\epsilon(\tau,\vs)=\sum_{i=1}^N \epsilon_i(\tau,s_i),
$$
and consider the auxiliary function $\psi \in \Ccal^1(\Gamma,\RR)$ defined for all $\vs\in\Gamma$ by 
$$
\psi(\vs)=\epsilon({\bar \tau}(\vs),\vs)-\epsilon(\tau_L,\vs)+H_L({\bar \tau}(\vs)-\tau_L)-
\frac{m^2}{2}({\bar \tau}(\vs)^2-\tau_L^2).
$$
Here we set
$$
H_L = m^2 \tau_L + p(\tau_L,\vs_L),
$$
so that $\Hcal(\tau,\vs)$ recasts as 
$$
\Hcal({\bar \tau}(\vs),\vs)=\psi(\vs)+\epsilon(\tau_L,\vs)-\epsilon(\tau_L,\vs_L).
$$

Arguing about the identity $\Fcal(\bar{\tau}(\vs),\vs)=0$ valid for
all $\vs\in\Gamma$ (see Lemma \ref{lem1_cp}.(ii)), we have 
\be\label{starl6}
H_L=p(\bar{\tau}(\vs),\vs)+m^2\bar{\tau}(\vs),\nonumber\\
m^2(\bar{\tau}(\vs)-\tau_L)=p(\tau_L,\vs_L)-p(\bar{\tau}(\vs),\vs)\nonumber,
\ee
which gives successively for all $\vs\in\Gamma$:
\be
\nonumber
\begin{aligned}
\psi(\vs)&=&\epsilon(\bar{\tau}(\vs),\vs)-\epsilon(\tau_L,\vs)+(\bar{\tau}(\vs)-\tau_L)\left(p(\bar{\tau}(\vs),\vs)
+\frac{m^2}{2}(\bar{\tau}(\vs)-\tau_L)\right),\\
&=& \epsilon(\bar{\tau}(\vs),\vs)-\epsilon(\tau_L,\vs)+(\bar{\tau}(\vs)-\tau_L)\left(p(\bar{\tau}(\vs),\vs)-
\frac{1}{2}(p(\bar{\tau}(\vs),\vs)-p(\tau_L,s_L))\right).
\end{aligned}
\ee
Moreover, the two identities $\Fcal({\bar \tau}(\vs),\vs)=0$ and
$\frac{\partial\Fcal}{\partial\tau}(\bar{\tau},\vs)=0$ valid for all
$\vs\in\Gamma$ are easily seen to give for the $\vs$ under
consideration:
$$
p(\bar{\tau}(\vs),\vs)-p(\tau_L,\vs_L)=(\bar{\tau}(\vs)-\tau_L)
\frac{\partial p}{\partial\tau}(\bar{\tau}(\vs),\vs).
$$

Consequently, for all $\vs\in\Gamma$
$$
\psi(\vs)=\epsilon(\bar{\tau}(\vs),\vs)-\epsilon(\tau_L,\vs)+(\bar{\tau}(\vs)-\tau_L)p(\bar{\tau}(\vs),\vs)
-\frac{1}{2}(\bar{\tau}(\vs)-\tau_L)^2\frac{\partial p}{\partial\tau}(\bar{\tau}(\vs),\vs).
$$
To conclude, we show that
\be
\label{chbtp}
\Hcal(\bar{\tau}(\vs),\vs)=\theta(\bar{\tau}(\vs),\vs)+\epsilon(\tau_L,\vs)-\epsilon(\tau_L,\vs_L) > 0.
\ee
Since $\frac{\partial\epsilon_i}{\partial s_i}(\tau_L,s_i)=T_i(\tau_L,s_i)>0$ then
$\epsilon(\tau_L,\vs)-\epsilon(\tau_L,\vs_L)>0$ for all $s\in\Gamma$ since $\vs_L\not\in\Gamma$.
Indeed, observe that Lemma \ref{lem2_cp} implies that equality to
zero holds iff $\vs=\vs_L$ but $\vs_L\not\in\Gamma$.

To show (\ref{chbtp}), we study the following auxiliary function
$\Psi \in \Ccal^1\big(\RR^*_+,\RR\big)$, setting for fixed
$s\in\Gamma$:
$$
\Psi(\tau)=\epsilon({\tau},\vs)-\epsilon(\tau_L,\vs)+({\tau}-\tau_L)p({\tau},\vs)
-\frac{1}{2}({\tau}-\tau_L)^2\frac{\partial p}{\partial\tau}({\tau},\vs).
$$
Easy calculations give 
$$
\frac{\partial\Psi}{\partial\tau}(\tau)=-\frac{1}{2}(\tau-\tau_L)^2\frac{\partial^2 p}{\partial\tau^2}
(\tau,\vs)\le 0,
$$
with $\Psi(\tau_L)=0$. Consequently, $\Psi(\tau)>0$ for all $\tau<\tau_L$. Since 
$\bar{\tau}(\vs)<\tau_L$ (see Lemma \ref{lem1_cp}) then $\Psi(\bar{\tau}(\vs))>0$ for all $\vs\in\Gamma$, and 
we thus obtain the required inequality: $\Hcal(\bar{\tau}(\vs),\vs)>0$.
\end{proof}

\begin{proof}[Proof of Proposition \ref{prop2_cp}]
We first establish the required properties of $\Hcal$ related to
the branch of solutions $\tau^+$. Arguing about the identity $\tau^+(\vs)={\bar \tau}(\vs)$
for all $\vs\in\Gamma$, the technical Lemma \ref{lem3_cp} allows to
restrict ourselves to $\vs\in {\bar \Dcal}\backslash \Gamma$ where
$\tau^+$ is continuously differentiable. For such $\vs$, the identity
$\Fcal(\tau^+(\vs),\vs)=0$ re-expresses equivalently:
\begin{eqnarray}
m^2\big(\tau^+(\vs)-\tau_L\big)=\left( p(\tau_L,\vs_L)-p(\tau^+(\vs),\vs)\right),\label{l7b}.
\end{eqnarray}

Let us evaluate $\Hcal(\tau^+(\vs),\vs)$ as follows:
$$
\Hcal(\tau^+(\vs),\vs)=\epsilon(\tau^+(\vs),\vs)-\epsilon(\tau_L,\vs_L)
+(\tau^+(\vs)-\tau_L)\!\!\left(p(\tau^+(\vs),s)+\frac{m^2}{2}(\tau^+(\vs)-\tau_L)\right)
$$
where $\epsilon(\tau,\vs)=\sum_{i=1}^N \epsilon_i(\tau,s_i)$. Using (\ref{l7b}), we 
then  obtain 
$$
\Hcal(\tau^+(\vs),\vs)=\epsilon(\tau^+(\vs),\vs)-\epsilon(\tau_L,\vs_L)
-\frac{1}{2m^2}\left(p^2(\tau^+(\vs),\vs)-p^2(\tau_L,\vs_L)\right).
$$
Let us introduce the auxiliary function $\Theta : \RR^*_+\times{\bar \Dcal} \to \RR$ 
by setting 
$$
\Theta(\tau,\vs)=\epsilon(\tau,\vs)-\frac{1}{2m^2}p^2(\tau,\vs),
$$
so that for all $\vs\in{\bar \Dcal}$:
\be
\label{rt1_cp}
\Hcal(\tau^+(\vs),\vs)=\Theta(\tau^+(\vs),\vs)-\Theta(\tau_L,\vs_L),
\ee
with $\vs\mapsto \Theta(\tau^+(\vs),\vs) \in \Ccal^1({\bar \Dcal}\backslash\Gamma,\RR)$.
Since $\tau^+(\vs_L)=\tau_L$ by Lemmas~\ref{lem2_cp} and \ref{rt1_cp} reads equivalently:
$$
\Hcal(\tau^+(\vs),\vs)=\Theta(\tau^+(\vs),\vs)-\Theta(\tau^+(\vs_L),\vs_L).
$$

Moreover, for all $\vs\in{\bar \Dcal}\backslash \Gamma$ we have:
\begin{eqnarray}
\frac{\partial}{\partial s_i}\theta(\tau^+(\vs),\vs)
&=& \frac{\partial\tau^+}{\partial s_i}(\vs)\frac{\partial \Hcal}{\partial\tau}(\tau^+(\vs),\vs)+
   \frac{\partial \Hcal}{\partial s_i}(\tau^+(\vs),\vs),\nonumber\\
&=& -\frac{\partial\tau^+}{\partial s_i}(\vs)\Fcal(\tau^+(\vs),\vs)+
   \frac{\partial \epsilon_i}{\partial s_i}(\tau^+(\vs),\vs)\nonumber\\
&=&\frac{\partial \epsilon_i}{\partial s_i}(\tau^+(\vs),\vs)= T_i(\tau^+(\vs),\vs)>0,
\end{eqnarray}
where we used the identity $\Fcal(\tau^+(\vs),\vs)=0$. Consequently, we deduce that
$$
\theta(\tau^+(\vs),\vs)-\theta(\tau^+(\vs_L),\vs_L)\ge 0,~~\forall \vs\in{\bar \Dcal}\backslash\Gamma
$$
with equality to zero iff $\vs=\vs_L$ (see Lemma
\ref{lem2_cp}). Combining the previous steps with Lemma \ref{lem3_cp}
gives the required properties $(i)$ and $(ii)$.

We now derive the remaining properties of $\Hcal$ related to
$\tau^-$. Observe that the technical Lemma \ref{lem3_cp}
immediately gives 
\be
\label{tmG}
\Hcal(\tau^-(\vs),\vs) > 0, \quad   \vs\in\Gamma,
\ee
since $\tau^-(\vs)={\bar \tau}(\vs)$ for the $\vs$ under consideration. 
We can now 
obtain the following estimate 
$$
\label{tmsl}
\Hcal(\tau^-(\vs_L),\vs_L) < 0.
$$
To that purpose, let us introduce the following auxiliary function $\psi : \RR^*_+ \to \RR$ setting:
$$
\psi(\tau) = \Hcal(\tau,\vs_L).
$$
Since $\psi'(\tau) = -\Fcal(\tau,\vs_L)$ for all $\tau >0$, Lemma \ref{lem2_cp} is easily seen to imply that 
$\psi$ strictly increases in $\Big(\tau^-(\vs_L),\tau^+(\vs_L)\big)$ with $\Hcal(\tau^+(\vs_L),\vs_L)=0$ as  
just established. This yields inequality (\ref{tmsl}).

To conclude the proof, we follow exactly the same steps as those
developed in the proof of Lemma \ref{lem1_cp} devoted to the
derivation of the subset $\Kcal\in S^+(\vs_L)$.

We introduce the following auxiliary function defined by 
$$
\Phi(\lambda,\va) = \Hcal\big(\tau^-(\vs_L + \lambda \va), \vs_L+\lambda \va\big), \quad 
 \quad \va\in \Scal^N_+,\quad \lambda \in [0,{\bar \Lambda}(\va)[.
$$
Note that this function is continuously differentiable on its domain of definition 
since, in view of Lemma \ref{lem2_cp}, the function
$(\tau,\lambda)\mapsto\tau^-(\vs_L+\lambda \va)$ is differentiable.
Straightforward calculations then give
\be
\label{derPhi}
\aligned
{{\partial \Phi}\over{\partial \lambda}}(\lambda,\va) 
= &
-\Fcal(\tau^-(\vs_L+\lambda \va),\vs_L+\lambda \va)\left( \sum_{1\le i\le N} 
{{\partial \tau^-}\over{\partial s_i}}a_i\right) 
\\
&
+ \sum_{1\le i\le N} T_i\left(\tau^-(\vs_L+\lambda \va),\vs_L+\lambda \va\right)a_i.
\endaligned
\ee
But the following identity $\Fcal(\tau^-(\vs_L+\lambda
\va),\vs_L+\lambda \va) = 0$ holds true by definition for all $\va\in
\Scal^N_+$ and $\lambda \in [0,{\bar\Lambda}(\va)]$ so that
(\ref{derPhi}) reduces to 
$$
{{\partial \Phi}\over{\partial \lambda}}(\lambda,\va) = \sum_{1\le i\le N} T_i(\tau^-(\vs_L+\lambda \va),\vs_L+\lambda \va)a_i > 0.
$$
Arguing about the inequalities (\ref{tmG}) and (\ref{tmsl}), the
implicit function theorem implies the existence of a map $\Lambda_0
\in \Ccal^1(\Scal^N_+,\RR^*_+)$ with the following properties:
$$
\Phi(\Lambda_0(\va),\va) = 0, \quad  \quad \va\in \Scal^N_+,
$$
together with $\Phi(\lambda,\va) < 0$ for all $\lambda \in [0,\Lambda_0(\va)[$ and $\Phi(\lambda,\va)>0$ for all
$\lambda\in ]\Lambda_0(\va),{\bar \Lambda}(\va)]$. This concludes the proof of Proposition \ref{prop2_cp}.
\end{proof}

\begin{proof}[Proof of Proposition \ref{prop3_cp}]
Arguing about the identity
$\partial_\tau\Hcal(\tau,\vs)=-\Fcal(\tau,\vs)$ valid for all
$(\tau,\vs)\in \RR^*_+\times\Kcal$,
Lemma \ref{lem2_cp} immediately implies that the smooth map $\tau
\mapsto \Hcal(\tau,\vs)$, $\vs$ being fixed in $\Kcal$,
strictly decreases in $]0,\tau^-(s)[$ and $]\tau^+(s),+\infty[$ while
it strictly increases in $]\tau^-(\vs),\tau^+(\vs)[$ with the
following limits $\lim_{\tau\to 0^+}\Hcal(\tau,\vs) = +\infty$ and
$\lim_{\tau\to\infty}\Hcal(\tau,\vs)=-\infty$ in view of the
asymptotic conditions \ref{thermo1}. In addition, for all $\vs\in\Kcal$,
we infer from Proposition \ref{prop2_cp} that $\Hcal(\tau^-(\vs),\vs) <
0$ and $\Hcal(\tau^+(\vs),\vs) > 0$. These observations allow
the definition of three maps, namely ${\check \Tcal}$, $\Tcal$, ${\hat
  \Tcal} : \Kcal \to \RR^*_+$ with the following properties:
$$
\Hcal({\check \Tcal}(\vs),\vs) = \Hcal({ \Tcal}(\vs),\vs) =
\Hcal({\hat \Tcal}(\vs),\vs) = 0, \quad  \quad \vs\in
\Kcal,
$$
together with 
$$
0 < {\check \Tcal}(\vs) < \tau^-(\vs) < { \Tcal}(\vs) < \tau^+(\vs) <
{\hat \Tcal}(\vs), \quad  \quad \vs\in \Kcal.
$$

Next, using the notation introduced in the proof of Lemma \ref{lem1_cp}, we first compute for all $\va\in \Scal^N_+$:
$$
\lim_{\lambda\to 0^+}{\check \Tcal}(\vs_L+\lambda\va) < \lim_{\lambda
  \to 0^+} { \Tcal}(\vs_L + \lambda \va) = \lim_{\lambda \to 0^+}
    {\hat \Tcal}(\vs_L + \lambda \va) = \tau^+(\vs_L),
$$
since $\Hcal(\tau^-(\vs_L),\vs_L) < \Hcal(\tau^+(\vs_L),\vs_L) = 0$ in
view of $(iii)$ and $(i)$ in Proposition \ref{prop2_cp}.
In the same way, we get 
$$
\aligned
\lim_{\lambda\to \Lambda_0(\va)}{\check \Tcal}(\vs_L+\lambda\va) 
& = \lim_{\lambda \to\Lambda_0(\va)} { \Tcal}(\vs_L + \lambda \va) 
\\
& = \tau^-(\vs_L+\Lambda_0(\va)\va) < \lim_{\lambda \to \Lambda_0(\va)} {\hat \Tcal}(\vs_L + \lambda \va),
\endaligned
$$
since $\Hcal(\tau^-(\vs),\vs) = 0 < \Hcal(\tau^+(\vs),\vs)$ for all
$\vs\in \Ccal$, in view of Proposition \ref{prop2_cp}. To conclude, we
have to establish the smoothness properties put forward in
Proposition \ref{prop3_cp}. In view of the monotonicity properties of
$\tau \mapsto \Hcal(\tau,\vs)$ we have just established for all
$\vs\in\Kcal$, all the three maps are obviously in
$\Ccal^1\big(\Kcal,\RR^*_+\big)\cap\Ccal^0\big({\bar
  \Kcal},\RR^*_+\big)$ thanks to the implicit function theorem.
This concludes the proof of Proposition~\ref{prop3_cp}.
\end{proof}


\section*{Acknowledgments}   

The authors were supported by a DFG-CNRS collaborative grant between France and Germany
on ``Micro-Macro Modeling and Simulation of Liquid-Vapor Flows''.  
The third author was also supported 
by the Centre National de la Recherche Scientifique (CNRS) and 
the Agence Nationale de la Recherche (ANR) via the grant 06-2-134423.


\newcommand{\auth}{\textsc}

\end{document}